\documentclass[journal,twoside]{IEEEtran}

\usepackage[T1]{fontenc}
\usepackage[utf8]{inputenc}	

\usepackage{amsmath,environ}
\usepackage{amsfonts}
\usepackage{mathtools}

\newcommand{\tran}{^{\mathstrut\scriptscriptstyle\top}}
\usepackage{xifthen}

\usepackage[pdfpagelabels]{hyperref}

\usepackage[nice]{units}
\usepackage[nice]{nicefrac}

\makeatletter
\newcommand{\subalign}[1]{%
  \vcenter{%
    \Let@ \restore@math@cr \default@tag
    \baselineskip\fontdimen10 \scriptfont\tw@
    \advance\baselineskip\fontdimen12 \scriptfont\tw@
    \lineskip\thr@@\fontdimen8 \scriptfont\thr@@
    \lineskiplimit\lineskip
    \ialign{\hfil$\m@th\scriptstyle##$&$\m@th\scriptstyle{}##$\crcr
      #1\crcr
    }%
  }
}
\makeatother

\newcommand{\R}{\mathbb{R}}				
\newcommand{\N}{\mathbb{N}}				

\newcommand{\E}{\mathbf{E}}

\newlength\figureheight
\newlength\figurewidth
\newlength\ImageToCaption

\newcommand{\diag}[1]{\operatorname{diag}\left(#1\right)}

\newcommand{\setpoint}[1]{\ifthenelse{\equal{#1}{}}{\ensuremath{u(k)}}{\ensuremath{u_{\mathrm{#1}}(k)}}}
\newcommand{\switchState}[1]{\ifthenelse{\equal{#1}{}}{\ensuremath{\delta(k)}}{\ensuremath{\delta_{\mathrm{#1}}(k)}}}
\newcommand{\disturbance}[1]{\ifthenelse{\equal{#1}{}}{\ensuremath{w(k)}}{\ensuremath{w_{\mathrm{#1}}(k)}}}
\newcommand{\state}[1]{\ensuremath{x(k)}}
\newcommand{\power}[1]{\ifthenelse{\equal{#1}{}}{\ensuremath{p(k)}}{\ensuremath{p_{\mathrm{#1}}(k)}}}
\newcommand{\powerFlow}[1]{\ifthenelse{\equal{#1}{}}{\ensuremath{p_{\mathrm{e}}(k)}}{\ensuremath{p_{\mathrm{e},#1}(k)}}}
\newcommand{\powerSharingFactor}[1]{\ifthenelse{\equal{#1}{}}{\ensuremath{\chi}}{\ensuremath{\chi_{\mathrm{#1}}}}}
\newcommand{\smallM}[1]{\ensuremath{\mathrm{m}_\mathrm{#1}}}
\newcommand{\bigM}[1]{\ensuremath{\mathrm{M}_\mathrm{#1}}}

\newcommand{\One}{1}
\newcommand{\Zero}{0}
\newcommand{\Eye}{I}

\newcommand{\samplingTime}{\ensuremath{\mathrm{T}_{\mathrm{s}}}}
\newcommand{\frequency}{\ensuremath{\rho}}

\newcommand{\predictedTime}[1]{\ifthenelse{\equal{#1}{}}{\ensuremath{k|k}}{\ensuremath{k+#1|k}}}

\let\min\relax
\let\max\relax
\DeclareMathOperator*{\min}{{min}}
\DeclareMathOperator*{\max}{{max}}
\newcommand{\Max}{{\mathrm{max}}}
\newcommand{\Min}{{\mathrm{min}}}
\DeclareMathOperator*{\Minimize}{\mathbf{Minimise}}
\DeclareMathOperator{\subjectto}{\mathbf{subj.\ to}}

\DeclareMathOperator{\nodes}{\mathbf{nodes}}
\DeclareMathOperator{\stage}{\mathbf{stage}}
\DeclareMathOperator{\child}{\mathbf{child}}
\DeclareMathOperator{\ancestor}{\mathbf{anc}}

\newcommand{\AVaR}[0]{%
  \mathrm{AV@R}_{\alpha}%
}

\NewEnviron{alignsubeq}[1]{
	\begin{subequations}
		\begin{align}
			\BODY
		\end{align}
		\label{#1}
	\end{subequations}
}

\usepackage{amsthm}
\usepackage[capitalize]{cleveref}
\theoremstyle{definition}
\newtheorem{problem}{Problem}
\newtheorem{theorem}{Theorem}[section]
\newtheorem{lemma}[theorem]{Lemma}
\newtheorem{remark}[theorem]{Remark}

\newtheorem{assumption}[theorem]{Assumption}
\newtheorem{proposition}[theorem]{Proposition}

\usepackage{csquotes}

\usepackage[]{threeparttable}
\usepackage{graphicx}
\usepackage{wrapfig}
\usepackage{tabularx}
\usepackage[free-standing-units]{siunitx}
\usepackage{booktabs}
\usepackage{multirow}
\usepackage{multicol}
\usepackage{dcolumn}
\newcolumntype{d}[1]{D{.}{.}{#1}}

\newcommand{\noDecimal}[1]{\multicolumn{1}{c}{\centering #1}}

\usepackage{csvsimple}

\usepackage{acronym}
\acrodef{mg}[MG]{microgrid}
\acrodef{mpc}[MPC]{model predictive control}
\acrodef{res}[RES]{renewable energy sources}
\acrodef{ac}[AC]{alternating current} \acused{ac}
\acrodef{dc}[DC]{direct current} \acused{dc}
\acrodef{arima}[ARIMA]{autoregressive integrated moving average}
\acrodef{help}[HELP]{high-effect low-probability}
\acrodef{avar}[AVaR]{average value-at-risk}
\acrodef{miqcqp}[MIQCQP]{mixed-integer quadratically-constrained quadratic problem}

\begin{document}

\title{Risk-Averse Model Predictive Operation Control\\ of Islanded Microgrids}

\author{Christian A. Hans, \IEEEmembership{Student Member, IEEE}, Pantelis Sopasakis, Jörg Raisch, \\
Carsten Reincke-Collon and Panagiotis Patrinos, \IEEEmembership{Member, IEEE}
\thanks{C.\,A. Hans is with Technische Universit\"at Berlin, Control Systems Group, Germany
	(e-mail: hans@control.tu-berlin.de).}%
\thanks{P. Sopasakis is with Queen's University Belfast, School of Electronics, Electrical Engineering and Computer Science, Centre For Intelligent Autonomous Manufacturing Systems (i-AMS), Belfast, Northern Ireland, UK
	(e-mail: p.sopasakis@qub.ac.uk).}%
\thanks{J. Raisch is with Technische Universit\"at Berlin, Control Systems Group and Max-Planck-Institut f\"ur Dynamik komplexer technischer Systeme, Germany
	(e-mail: raisch@control.tu-berlin.de).}%
\thanks{C.~Reincke-Collon is with Aggreko plc, Germany (e-mail: carsten.reincke-collon@aggreko.com).}%
\thanks{P. Patrinos is with KU Leuven, Department of Electrical Engineering (ESAT), Belgium
	(e-mail: panos.patrinos@kuleuven.be).}%
\thanks{This work was partially supported by the German Federal Ministry for Economic Affairs and Energy (BMWi), projects No. 0324024A and 0325713A;
FWO projects No. G086318N and G086518N;
Fonds de la Recherche Scientifique--FNRS, Fonds Wetenschappelijk Onderzoek--Vlaanderen under EOS project No. 30468160 (SeLMA);
Research Council KU Leuven C1 project No. C14/18/068.}%
}

\maketitle

\begin{abstract}
	In this paper we present a risk-averse model predictive control (MPC) scheme for the operation of islanded microgrids with very high share of renewable energy sources.
	The proposed scheme mitigates the effect of errors in the determination of the probability distribution of renewable infeed and load.
	This allows to use less complex and less accurate forecasting methods and to formulate low-dimensional scenario-based optimisation problems which are suitable for control applications.
	Additionally, the designer may trade performance for safety by interpolating between the conventional stochastic and worst-case MPC formulations.
	The presented risk-averse MPC problem is formulated as a mixed-integer quadratically-constrained quadratic problem and its favourable characteristics are demonstrated in a case study.
	This includes a sensitivity analysis that illustrates the robustness to load and renewable power prediction errors.
\end{abstract}



\section{Introduction}
\label{sec:introduction}
The substitution of conventional power plants by \ac{res} is a key element in the fight against climate change.
However, it presents major challenges.
The structure of power supply is expected to change from a small number of large-scale power plants to a large number of small-scale units.
These will be geographically distributed over the entire electric grid.
Additionally, the uncertainty in power supply of some \ac{res} will complicate the operation of the grid.

One way to tackle these challenges is to partition the electric grid into \acp{mg} \cite{HAIM2007}.
These comprise storage, conventional and renewable units connected to each other and to loads by transmission lines \cite{Las2002}.
\acp{mg} can be operated connected to the grid or electrically isolated (islanded) \cite{LMM2006}.
Inspired by conventional power systems, hierarchical approaches have been promoted for the control of \acp{mg}, e.g., in \cite{BD2012}.
On the lower control layer, typically on a timescale from milliseconds to seconds, primary control aims to provide voltage and frequency stability.
Secondary control, typically on a timescale from seconds to minutes, aims at compensating frequency deviations and ensures that the voltages remain close to the desired values.
Operation control, also referred to as energy management, typically acts on a timescale of minutes to fractions of hours.
It aims at optimising the \ac{mg} operation by providing power setpoints to the units \cite{KIHD2008}.
For this task, \ac{mpc} approaches are considered a good choice as they allow to explicitly include constraints on the units and take into account the system dynamics.
Moreover, they can be combined with forecasts of load and renewable infeed to operate the \ac{mg} in an optimal way.

\subsection{Operation control of \acp{mg}}
%
%
Several approaches for the operation control of \acp{mg} have been proposed.
One way to categorise them is by the way they handle uncertainties.
Prominent formulations are
(i)~{certainty-equivalent}, where a deterministic system model is fully trusted,
(ii)~{worst-case}, where no probability distribution is considered,
(iii)~{risk-neutral stochastic}, where an underlying probability distribution is trusted, and
(iv)~{risk-averse}, where the underlying distribution is not fully trusted.

%
%
There is a variety of publications on operation control of \acp{mg} where a perfect forecast is assumed.
For example, in \cite{TH2011} a certainty-equivalent approach for grid-connected \acp{mg} is proposed.
Based on the assumption that the forecast generation of \ac{res}, load and the energy price are certain, in \cite{PRG2014} an \ac{mpc} is presented.
For islanded \acp{mg}, an \ac{mpc} approach that also assumes perfect forecasts is given in \cite{PBL+2013}.
A certainty-equivalent approach that includes power flow over the lines is proposed in \cite{OCK2014}.
However, as shown in \cite{HNRR2014}, in the operation of islanded \acp{mg} with high share of \ac{res}, certainty-equivalent approaches can lead to significant constraint violations.

%
%
To compensate for this lack of robustness, some authors have proposed worst-case \ac{mpc} approaches for the operation of \acp{mg}.
Uncertain parameters, such as model parameters \cite{Prodan2014399} and forecasts \cite{HNRR2014,HNRR2015}
are treated as bounded quantities and an \ac{mpc} formulation is used to minimise the worst-case cost.
The \ac{mpc} approaches for islanded \acp{mg} in \cite{HNRR2014,HNRR2015} include power flow over the lines, power sharing of grid-forming units and curtailable renewable infeed.
However, these approaches have been found to be overly conservative as they try to minimise the worst-case objective \cite{HNRR2014,HNRR2015}.

%
%
The conservativeness of worst-case approaches has led to the wide adoption of stochastic methods, which consist in minimising the expected value of a random cost with respect to an assumed probability distribution.
There are two main approaches to modelling randomness:
(i)~random processes with continuous distributions and
(ii)~scenario trees.

%
%
In the area of microgrid control, some authors have used the assumption that all involved uncertain quantities are normally distributed~\cite{KouLiang+2016,Gulin+2015,Cominesi+2018,HHM+2012}.
Having unbounded support, Gaussian disturbances make it impossible to satisfy state constraints uniformly;
instead chance constraints can be used.
Dealing with chance constraints hinges on the normality assumption and requires that the involved random processes be independent \cite{MesbahReview2016,OJM2008}.
An approach for chance constraints that drops the normality assumption using polynomial chaos expansions and machine learning is presented in~\cite{KouLiangGao2018}.%

%
%
Scenario-based approaches seem to be more
popular in microgrid control \cite{PRG2014,HHM+2012,HSB+2015,Petrollese201696,PRG2016,MD2017,HBSJ2016}; mainly because scenario
trees can be constructed from data \cite{Pinson+2009,HR2003,HeiRom09}
and the assumption of independence does not need to be imposed.
In \cite{Petrollese201696}, a scenario-based approach combining an optimal operation scheduling with an \ac{mpc} and assuming uncertain weather and load was proposed.
Furthermore, \cite{PRG2014} was extended in \cite{PRG2016} to a two-stage stochastic \ac{mpc} approach, which was formulated using scenario trees.
In \cite{MD2017}, a scenario-based optimal operation control strategy for droop controlled \acp{mg} is presented, where a heuristic particle swarm optimisation is used to minimise the expected objective value while accounting for power limitations of the transmission lines.
A stochastic continuous-time rolling horizon control strategy that considers uncertain load can be found in \cite{HBSJ2016}.
However, this approach disregards the power flow over the transmission lines, possible power sharing among grid-forming units and the possibility to limit the power provided by \ac{res}.
There exist other scenario-based approaches, e.g., \cite{HHM+2012}, that account for power flow.
However, \cite{HHM+2012} does not allow to limit infeed from \ac{res}.
One approach that includes power flow, curtailable \ac{res} and power sharing is the scenario-based stochastic \ac{mpc} presented in \cite{HSB+2015}.

Scenario-based \ac{mpc} approaches that minimise the expectation of the cost assume that the scenario tree offers an adequate representation of the underlying probability distribution.
The realisation that this assumption does not always hold led to the emergence of risk-averse MPC \cite{HSBP2017,CP2014} which aims at safeguarding the controlled system from the effects of inexact knowledge of the underlying probability distribution.

\subsection{Risk-based approaches in power systems}

Risk measures stem from the domains of operations research, stochastic finance and actuarial science~\cite{SDR2014}.
An example of a popular risk measure is the \ac{avar} \cite{SDR2014}.
When used in optimisation problems, risk measures allow to mitigate the effects of inexact knowledge of the underlying probability distributions.
They bridge the gap between conservative worst-case approaches, which assume no knowledge about the underlying probability distribution and stochastic approaches, which assume perfect knowledge about it.
Therefore, risk-averse approaches are suitable for practical implementations where probability distributions are not known exactly.
In the case of operation management of \acp{mg} it is likely that the predicted load and renewable infeed are uncertain and inaccurately estimated probability distribution.
Such errors in the probability distribution can have a high impact on the controlled system.
Therefore, it is important to design controllers that are robust with respect to forecast errors and errors in the determination of their probability distribution.

Although risk-averse problems enjoy favourable properties, their applicability has been limited by their complexity and computational cost of resulting multistage risk-averse optimal control problems.
The reason is that their cost function is a composition of several nonsmooth mappings \cite{SDR2014}.
To date, there have only been slow numerical optimisation methods (e.g., stochastic dual dynamic programming) and limited to linear cost functionals \cite{Asamov2015,Bruno2016979}.
An alternative solution approach uses multiparametric piecewise quadratic programming \cite{patrinos2011convex}, yet its applicability is limited to systems with few states and small prediction horizons \cite{patrinosECC2007}.
This has been daunting and risk-averse problems involving integer variables are considered overly computationally complex for real-time applications.

The study of risk-averse optimisation problems is gaining popularity in power systems applications \cite{KS2016} such as unit commitment \cite{ZJ2018}, scheduling \cite{GPM2007,MNF+2013} and optimal power flow \cite{ZSM2017,SWML2015,Roald201566}.
In \cite{Maceira2015126}, using \ac{avar}, an approach for expansion and operation planing for a hydrothermal system is proposed.
A multi-stage formulation for short-term trading with uncertain power from wind turbines and market prices is introduced in \cite{MCP2010}.
For grid-connected \acp{mg} a comparison between a scenario-based risk-averse stochastic and a worst-case optimisation approach is presented in \cite{KS2016}.

\begin{figure*}[t]
	\centering
	\includegraphics{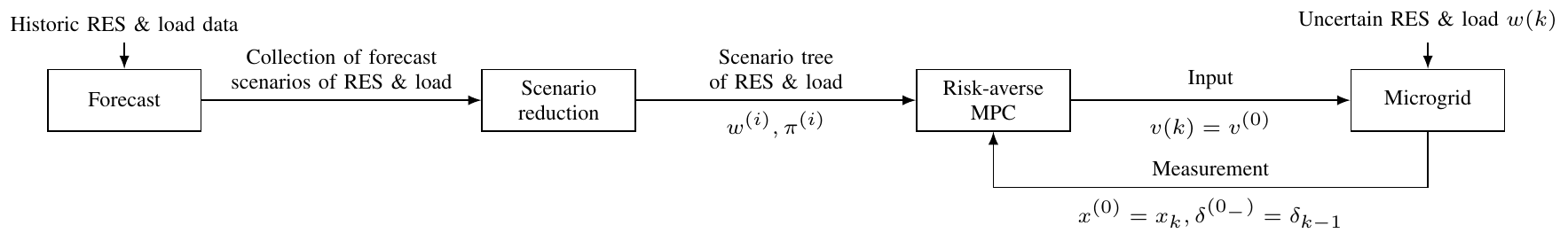}%
	\caption{Control scheme used in the risk-averse \acs{mpc} approach for islanded \acsp{mg} (motivated by \cite{HSB+2015,PTB2011}).}
	\label{fig:systemOverview}
\end{figure*}

Despite their increasing popularity, most risk-averse approaches \cite{ZJ2018,MNF+2013,Maceira2015126,BZW+2012} are based on simple formulations which fail to capture how risk propagates in time.
Proper multistage formulations lead to intricate optimisation problems where the cost function is expressed as a composition of nonsmooth mappings.
This makes them less suitable for \ac{mpc}~\cite{HSBP2017}, where fast computations are required.
Therefore, in this paper we use a reformulation that decomposes these nested mappings and allows to solve risk-averse \ac{mpc} problems online.

\subsection{Contributions}
In brief, the contributions of this work are:
(i)~we introduce a novel constrained hybrid dynamical model for islanded \acp{mg},
(ii)~we propose a multi-stage risk-averse \ac{mpc} problem for the optimal operation of an \ac{mg},
(iii)~we reformulate the risk-averse optimal control problem in a computationally tractable way and
(iv)~we demonstrate the properties of the proposed operation control scheme in a simulation case study.

(i)~We present the model of an islanded \ac{mg} with uncertain renewable generation and loads that allows for configurations with very high share of \ac{res}.
This model, motivated by \cite{HNRR2014,HSB+2015}, considers a possible limitation of renewable infeed while limitations on transmission lines are approximately accounted for using \ac{dc} power flow approximations.
We model storage devices as grid-forming units and consider power sharing with the enabled conventional generators.
This way, the fluctuations of load and renewable generation affect the power of all units and the state of charge of the storage devices.
This allows to model operating modes where only \ac{res} and storage units are used and no conventional unit is required.

(ii)~We extend \cite{HNRR2014,HSB+2015} to formulate a risk-averse \ac{mpc} problem for islanded \acp{mg}.
Unlike risk-neutral stochastic \ac{mpc} \cite{HSB+2015,PTB2011}, where a large number of scenarios is needed
for an accurate representation of the distributions, risk-averse \ac{mpc} allows for fewer scenarios
as it mitigates the effect of uncertainty in the estimated probability distribution.
This allows for robustness against bad forecast models of load and renewable infeed, time-varying probability distributions, or approximation errors in the scenario tree generation.
Risk-averse formulations allow to interpolate between worst-case~\cite{HNRR2014} and stochastic~\cite{HSB+2015} \ac{mpc} to specify the acceptable level of risk and provide resilience against \acl{help} events.

(iii)~Motivated by \cite{HSBP2017}, we use an epigraphical relaxation to reformulate the original risk-averse problem as a \ac{miqcqp}.
This way we decompose the original nested formulation and render the problem formulation suitable for real-time applications.

(iv)~In a comprehensive case study, we demonstrate the use of the proposed risk-averse \ac{mpc} scheme for a simple \ac{mg}.
We juxtapose the operation of the \ac{mg} using a stochastic, a worst-case and a risk-averse formulation to show that the conservativeness of the controller can be tuned.
Lastly, the robustness with respect to uncertainties in the probability distribution of load and \ac{res} is investigated by means of a sensitivity analysis.

\subsection{Structure of paper}
\label{sec:structure}
The remainder of the paper is structured along the lines of \Cref{fig:systemOverview}.
In \Cref{sec:microgridModel} the model of an islanded \ac{mg} is introduced.
Then, scenario trees are derived from time-series based forecasts in \Cref{sec:forecastAndScenarioTree}.
In \Cref{sec:operatingCosts} we quantify the operating costs of the \ac{mg}.
Subsequently, risk measures are discussed in \Cref{sec:measuringRisk} and a risk-averse \ac{mpc} approach is derived in \Cref{sec:riskAverseMpc}.
Finally, in \Cref{sec:caseStudy}, the properties of the resulting \ac{mpc} are illustrated in a numerical case study.

\subsection{Notation}
Real and natural numbers are denoted by $\R$ and $\N$, respectively.
The set of nonnegative integers is denoted by $\N_{0}$.
The set $\{x| x\in\N_0 \wedge a \leq x \leq b\}$ is denoted by $\N_{[a,b]}$.
Furthermore, the set of nonnegative real numbers is $\R_{\geq 0}$ and the set of positive real numbers is $\R_{> 0}$.
The cardinality of a set $\mathbb{V}$ is denoted by $|\mathbb{V}|$.
The transpose of a vector $a$ is $a\tran$.
The vector $[a_{v_1} ~ a_{v_2} ~ \cdots ~ a_{v_N}]\tran$ composed of elements $a_{v_i}$ for all $v_i\in\mathbb{V} = \{v_1, v_2, \ldots, v_N\}\subset\N$ with $v_i < v_j$ for $i < j$, $i \in \N_{[1,N]}$, $j \in \N_{[1,N]}$ is denoted by $[a_i]_{i\in\mathbb{V}}$.
A vector of dimension $N$ whose elements are all equal to $1$ is denoted by $\One_N$.
Let $a = [a_1 ~\cdots ~ a_N]\tran$.
Then, $\diag{a}$ denotes the diagonal matrix with entries $a_i$, $i \in \N_{[1,N]}$.

When used with vectors, the relations $\geq$, $\leq$, $<$, $>$ are understood element-wise.
Function $\max(a, b)$ returns the element-wise maximum of the vectors $a$ and $b$.
However, if the function is used with only one vector input argument, i.e., $\max(a)$, then it returns the largest element of the vector $a$.
The same holds for the minimum function $\min(\cdot)$.
We denote by $\min_{x\in\mathbb{X}} f(x)$ the minimum value of function $f$ over $\mathbb{X}$ and by $\max_{x\in\mathbb{X}} f(x)$ the maximum value of $f$ over $\mathbb{X}$.


\section{Microgrid model}
\label{sec:microgridModel}

In this section, we derive a control-oriented mathematical model of an islanded \ac{mg}.
The model includes loads, conventional, renewable and storage units as well as transmission lines.
The example of a basic \ac{mg} that includes all these components is shown in \cref{fig:examplaryMicrogrid}.
The \ac{mg} model has the form
\begin{subequations}\label{eq:overall-model}
\begin{align}
	x(k) + {B} q(k) - x(k+1) &= 0, \label{eq:overall-model-dynamics}\\
	{H}_1 \cdot x(k+1) & \leq h_1, \label{eq:overall-model-dynamics-inequalities} \\
	{H}_2 \cdot \begin{bmatrix} v(k)\tran & q(k)\tran & w(k)\tran \end{bmatrix}\tran &\leq h_2,\label{eq:overall-model-inequalities} \\
	{G} \cdot \begin{bmatrix} v(k)\tran & q(k)\tran & w(k)\tran \end{bmatrix}\tran &= g,\label{eq:overall-model-algebraic}
\end{align}
\end{subequations}
where $k\in\N_{0}$ is the discrete time instant.
Here, $x(k)\in\R_{\geq 0}^S$ with $S\in\N$ is the state vector and $q(k)\in\R^Q$ is a vector of $Q\in\N$ auxiliary variables.
The control inputs are collected in the vector $v(k) = [\setpoint{}\tran ~ \switchState{t}\tran]\tran$, where
$\setpoint{} \in \R^U$ is the vector of real-valued control inputs of all $U \in \N$ units and
$\switchState{t} \in \{0, 1\}^T$ the vector of $T\in\N$ Boolean inputs.
The uncertain external inputs of the model are collected in the vector $w(k)\in \R^W$, $W \in \N$.
In \eqref{eq:overall-model-dynamics}, $B\in \R^{S \times Q}$ and in \eqref{eq:overall-model-dynamics-inequalities}, ${H}_1$, $h_1$ are of appropriate dimensions.
Furthermore, in \eqref{eq:overall-model-inequalities}, $H_2$ is a matrix and $h_2$ a vector of appropriate dimensions that reflect inequality constraints.
Likewise, in \eqref{eq:overall-model-algebraic} $G$ is a matrix and $g$ a vector of appropriate dimensions that reflect equality constraints.

In what follows, we will derive a control-oriented \ac{mg} model of the form \eqref{eq:overall-model}.
We start by posing some assumptions.

\begin{assumption}[Lower control layers]\label[assumption]{ass:microgridModel:lowerControlLayers}
The lower control layers, i.e., primary and secondary control, are designed such that the \ac{mg} can run autonomously for several minutes, so providing power setpoints to the units on the same timescale is sufficient.
These control layers ensure a desired power sharing (see, e.g., \cite{KHS+2017,SHK+2017}) among the grid-forming units.
\end{assumption}

\begin{assumption}[Conventional units]\label[assumption]{ass:microgridModel:conventionalUnits}
The startup and shutdown times of the conventional units are small compared to the sampling time of \ac{mpc}, i.e., switching actions are assumed to be instantaneous.
Changes in power are instantaneous, i.e., no climb rates need to be considered.
\end{assumption}

\begin{assumption}[Storage units]\label[assumption]{ass:microgridModel:storageUnits}
The state of charge can be estimated sufficiently accurately and is accessible to the operation control.
The error introduced by neglecting self discharge and conversion losses of storage units is small compared to the uncertainty introduced by renewable infeed and loads.
\end{assumption}

\begin{assumption}[Transmission lines]\label[assumption]{ass:microgridModel:transmissionLines}
The resistance of the transmission lines between the units and loads of the \ac{mg} as well as the reactive power flow are negligible.
The voltage amplitudes in the grid are constant and the phase angle differences small,
so the \ac{dc} power flow approximations \cite{PMVDB2005} can be used.
The error introduced hereby is small compared to the uncertainty introduced by renewable infeed and loads.
\end{assumption}

\subsection{Plant model interface}
The real-valued manipulated variables are the power setpoints of the units $\setpoint{} = [\setpoint{t}\tran ~ \setpoint{s}\tran ~ \setpoint{r}\tran ]\tran \in\R^U$ where $\setpoint{t}\in\R^T_{\geq0}$ are the setpoints of the $T$ conventional units, $\setpoint{s}\in\R^S$ the setpoints of the $S$ storage units and $\setpoint{r}\in\R^R_{\geq0}$ the setpoints of the $R$ \ac{res}.
Every conventional unit is associated with a Boolean input that indicates whether it is enabled or disabled.
All Boolean inputs are collected in a vector $\switchState{t}\in\{0, 1\}^T$.
Furthermore, the stored energies of the storage units are collected in the state vector $\state{}\in\R^{S}_{\geq 0}$.
The uncertain external inputs of the model are $\disturbance{} = [\disturbance{r}\tran ~ \disturbance{d}\tran]\tran$, where $\disturbance{r}\in\R^R_{\geq0}$ represents the maximum available power of the renewable units under given weather conditions and $\disturbance{d}\in\R^D_{\geq0}$ the load.

\subsection{Power of units}
In islanded mode, equilibrium of production, consumption and storage power must be ensured in presence of uncertain load and renewable infeed.
Therefore, the power of the units $\power{}\in\R^U$ is not necessarily equal to the setpoints $\setpoint{}$.
The vector of power values $\power{} = [\power{t}\tran ~ \power{s}\tran ~ \power{r}\tran]\tran$ is composed of the power of conventional units, $\power{t}\in\R^T_{\geq 0}$,
storage units, $\power{s}\in\R^S$, and
\ac{res}, $\power{r}\in\R^R_{\geq 0}$.

\begin{table}[t]
\centering
\caption{Model-specific variables}
\begin{tabular}{clcc}
  \toprule
  Symbol & Explanation & Unit & Size \\
  \midrule
  $x$ & Energy of storage units (state) & $\unit{pu\,h}$ & $S$\\
  \midrule
  $u_{\mathrm{t}}$ & Setpoint inputs of conventional units & $\unit{pu}$ & $T$ \\
  $u_{\mathrm{s}}$ & Setpoint inputs of storage units & $\unit{pu}$ & $S$\\
  $u_{\mathrm{r}}$ & Setpoint inputs of renewable units & $\unit{pu}$ & $R$ \\
  $u$ & Setpoint inputs of all units & $\unit{pu}$ & $U$ \\
  $\delta_{\mathrm{t}}$ & Boolean inputs of conventional units & --- & $T$ \\
  $v$ & Vector of all control inputs & --- & $Q$ \\
  \midrule
  $w_{\mathrm{r}}$ & Uncertain available renewable power & $\unit{pu}$ & $R$\\
  $w_{\mathrm{d}}$ & Uncertain load & $\unit{pu}$ & $D$\\
  $w$ & Vector of all uncertain inputs & $\unit{pu}$ & $W$ \\
  \midrule
  $p_{\mathrm{t}}$ & Power of conventional units & $\unit{pu}$ & $T$ \\
  $p_{\mathrm{s}}$ & Power of storage units & $\unit{pu}$ & $S$ \\
  $p_{\mathrm{r}}$ & Power of renewable units & $\unit{pu}$ & $R$ \\
  $p$ & Power of all units & $\unit{pu}$ & $U$ \\
  $p_{\mathrm{e}}$ & Power over transmission lines & $\unit{pu}$ & $E$ \\
  $\delta_{\mathrm{r}}$ & Boolean auxiliary variables & --- & $R$ \\
  $\rho$ & Real-valued auxiliary variable  & --- &$1$ \\
  $q$ & Vector of all auxiliary variables & --- & $Q$ \\
  \bottomrule
\end{tabular}
\end{table}

\subsubsection{\ac{res} units}
The power provided by the renewable units, $p_{\mathrm{r}}(k)$, and the corresponding setpoints, $u_{\mathrm{r}}(k)$, are limited by
\begin{subequations}\label{eq:microgridModel:renewablePower:limits}
\begin{alignat}{2}
	p_{\mathrm{r}}^{\min} & \leq \power{r}		&& \leq p_{\mathrm{r}}^{\max} \text{ and } \\
	p_{\mathrm{r}}^{\min} & \leq \setpoint{r} 	&& \leq p_{\mathrm{r}}^{\max}
\end{alignat}
with $p_{\mathrm{r}}^{\min}\in\R^R_{\geq 0}$ and $p_{\mathrm{r}}^{\max}\in\R^R_{\geq 0}$.
\end{subequations}

Additionally, the power infeed $p_{\mathrm{r},i}(k)\in\R_{\geq 0}$ of every renewable unit $i\in\N_{[1, R]}$ can be limited by the power setpoint $u_{\mathrm{r},i}(k)\in\R_{\geq 0}$.
However, the power only follows the setpoint if the maximum possible infeed under current weather conditions $w_{\mathrm{r},i}(k)\in\R_{\geq 0}$ is greater than or equal $u_{\mathrm{r},i}(k)$.
Using the element-wise $\min$ operator, this can be described by
\begin{equation} \label{eq:microgridModel:res:max}
	\power{r} = \min (\setpoint{r}, \disturbance{r}).
\end{equation}

\begin{subequations} \label{eq:microgridModel:renewablePower:bigM}%
For the formulation of the optimisation problem it is beneficial to transform \eqref{eq:microgridModel:res:max} into a set of linear inequalities involving integer variables.
This is done by introducing the free variable $\switchState{r}\in\{0,1\}^R$.
With the constants $\smallM{r}\in\R$, $\smallM{r} < \min(p_{\mathrm{r}}^{\min})$ and $\bigM{r}\in\R_{\geq 0}$, $\bigM{r} > \max(p_{\mathrm{r}}^{\max})$ which are calculated offline, we then exactly reformulate \eqref{eq:microgridModel:res:max} as (see, e.g., \cite{BM1999})%
\begin{align}%
	\power{r} &\leq \setpoint{r}, \\
	\power{r} &\geq \setpoint{r} + (\diag{\disturbance{r}} - \bigM{r} \Eye_{R}) \switchState{r}, \\
	\power{r} &\leq \disturbance{r}, \\
	\power{r} &\geq \disturbance{r} - (\diag{\disturbance{r}} - \smallM{r} \Eye_{R}) (\One_{R} - \switchState{r}).
\end{align}%
\end{subequations}%

\subsubsection{Conventional units}
The power provided by conventional unit $i\in\N_{[1, T]}$ is limited by $p_{\mathrm{t},i}^{\min}\in\R_{\geq0}$ and $p_{\mathrm{t},i}^{\max}\in\R_{\geq 0}$ if it is enabled, i.e., if $\delta_{\mathrm{t},i}(k) = 1$.
If the unit is disabled, i.e., $\delta_{\mathrm{t},i}(k) = 0$, then naturally $p_{\mathrm{t},i}(k)=0$.
In vector notation and with $p_{\mathrm{t}}^{\min}\in\R^T_{\geq 0}$, $p_{\mathrm{t}}^{\max}\in\R^T_{\geq 0}$ this can be expressed by
\begin{subequations}\label{eq:microgridModel:thermalUnit:limits}%
\begin{equation}
	\diag{p_\mathrm{t}^{\min}} \switchState{t} \leq \power{t} \leq \diag{p_\mathrm{t}^{\max}} \switchState{t}.
\end{equation}
The same holds for the power setpoints, i.e.,
\begin{equation}
	\diag{p_\mathrm{t}^{\min}} \switchState{t} \leq \setpoint{t} \leq \diag{p_\mathrm{t}^{\max}} \switchState{t}.
\end{equation}%
\end{subequations}%

\subsubsection{Storage units}
\begin{subequations}\label{eq:microgridModel:storageUnit:powerlimits}
As the storage units are assumed to be always enabled, all their setpoints and power values are limited by $p_{\mathrm{s}}^{\min}\in\R^S_{\leq 0}$ and $p_{\mathrm{s}}^{\max}\in\R^S_{\geq 0}$, i.e.,
\begin{alignat}{2}
p_{\mathrm{s}}^{\min} & \leq \power{s} && \leq p_{\mathrm{s}}^{\max},\\
 p_{\mathrm{s}}^{\min} & \leq \setpoint{s} && \leq p_{\mathrm{s}}^{\max}.
\end{alignat}%
\end{subequations}%

\subsection{Power sharing of grid-forming units}
Due to variations of load and renewable infeed, the power of all units does not necessarily match the power setpoints that are prescribed to the system.
The grid-forming units, i.e., all storage and conventional units, are assumed to be controlled by the lower control layers such that they share the changes in load and renewable infeed in a desired proportional manner.
This so called proportional power sharing (see, e.g., \cite{KHS+2017,SHK+2017}) depends on the design parameter $\chi_{i}\in\R_{> 0}$ for all grid-forming units.
A typical choice of $\chi_i$ is, e.g., proportional to the nominal power of the corresponding units.

Power sharing can be formalised as follows.
Two units ${i\in\N_{[1, T+S]}}$ and $j\in\N_{[1, T+S]}$, $i \neq j$ with $\chi_{i}\in\R_{>0}$ and $\chi_{j}\in\R_{>0}$ are said to share their power proportionally, if
\begin{equation}\label{eq:powerSharing:baseNotation}
	\frac{p_{i}(k) - u_{i}(k)}{\chi_i}
		=
	\frac{p_{j}(k) - u_{j}(k)}{\chi_j}
\end{equation}
holds.
Using the auxiliary free variable $\frequency(k)\in\R$ and including that only enabled units, i.e., units $i$ with $\delta_{\mathrm{t},i}(k) = 1$, can participate in power sharing, we can rewrite \eqref{eq:powerSharing:baseNotation} for all grid-forming units with
${K}_\mathrm{t} = \operatorname{diag}{([\nicefrac{1}{\chi_{1}} ~ \cdots ~ \nicefrac{1}{\chi_{T}}]\tran)}$ and
${{K}_\mathrm{s} = \operatorname{diag}([\nicefrac{1}{\chi_{(T+1)}} ~ \cdots ~ \nicefrac{1}{\chi_{(T+S)}}]\tran)}$ as
\begin{subequations}\label{eq:microgridModel:powerSharing}
\begin{align}
	\label{eq:microgridModel:powerSharing:thermal}
	{K}_\mathrm{t} (\power{t} - \setpoint{t}) &= \frequency(k) \switchState{t} \text{ and }\\
	\label{eq:microgridModel:powerSharing:storage}
	{K}_\mathrm{s} (\power{s} - \setpoint{s}) &= \frequency(k)  \One_{S}.
\end{align}
\end{subequations}
To proceed with what follows in the next sections, we need to transform \eqref{eq:microgridModel:powerSharing:thermal} into a set of linear inequalities with integer variables.
This can be done using a similar strategy as described in \cite{BM1999}.
First, we choose $\bigM{t}\in\R$ which can be calculated offline.
The value of $\bigM{t}$ should be greater than the biggest possible value of $\frequency(k)$.
Hence, with the biggest possible value
for the storage units, $\frequency_\mathrm{s}^{\max} = \max({K}_\mathrm{s} (p_{\mathrm{s}}^{\max} - p_{\mathrm{s}}^{\min}))$, and
for the conventional units, $\frequency_\mathrm{t}^{\max} = \max({K}_\mathrm{t} (p_{\mathrm{t}}^{\max} - p_{\mathrm{t}}^{\min}))$,
$\bigM{t}$ has to be chosen such that $\max{} \big( \frequency_\mathrm{s}^{\max} , \frequency_\mathrm{t}^{\max} \big) < \bigM{t}$.
Then, with $\smallM{t} = - \bigM{t}$, we can exactly reformulate \eqref{eq:microgridModel:powerSharing:thermal} as
\begin{subequations} \label{eq:microgridModel:powerSharing:thermal:bigM}
\begin{align}
	K_{\mathrm{t}} (\power{t} - \setpoint{t}) & \leq \bigM{t} \switchState{t}, \\
	K_{\mathrm{t}} (\power{t} - \setpoint{t}) & \geq \smallM{t} \switchState{t}, \\
	K_{\mathrm{t}} (\power{t} - \setpoint{t}) & \leq \One_{T} \frequency(k) - \smallM{t} (\One_T - \switchState{t}),\\
	K_{\mathrm{t}} (\power{t} - \setpoint{t}) & \geq \One_{T} \frequency(k) - \bigM{t} (\One_T - \switchState{t}).
\end{align}
\end{subequations}%

\subsection{Dynamics of storage units}
The dynamics of all storage units are assumed as
\begin{subequations}%
\label{eq:microgridModel:energy}%
\begin{equation} \label{eq:microgridModel:energy:dynamics}
	x(k+1) = x(k) -\samplingTime \power{s},
\end{equation}
where $\samplingTime\in\R_{>0}$ is the sampling time.
The stored energy is represented by $x(k)$ with initial state $x(0) = x_0$.
To cover for the limited storage capacity, $x(k+1)$ is bounded by
\begin{equation}
	x^{\min} \leq x(k+1) \leq x^{\max},
\end{equation}%
\end{subequations}%
with $x^{\min} = \Zero_{S}$ and $x^{\max}\in\R^{S}_{\geq 0}$.

\begin{remark}\label[remark]{rem:mgStorageModel}%
	In the simulations in \Cref{sec:caseStudy}, we use a more detailed plant model that is different from \eqref{eq:microgridModel:energy:dynamics}.
	This storage model is motivated by \cite{PRG2014} and includes self-discharge as well as charging and discharging efficiencies.
	It reads
	\begin{equation}\label{eq:microgridModel:energy:storageLosses}
		x(k+1) =
		\begin{cases}
			x(k) -\samplingTime \eta^{\mathrm{c}} p_{\mathrm{s}}(k) - x^{\mathrm{sd}},
				& \text{if } p_{\mathrm{s}}(k) \leq 0, \\
			x(k) -\samplingTime (\eta^{\mathrm{d}})^{-1} p_{\mathrm{s}}(k) - x^{\mathrm{sd}},
				& \text{if } p_{\mathrm{s}}(k) > 0.
		\end{cases}
	\end{equation}
	Here, $\eta^{\mathrm{c}}\in(0, 1]^{S\times S}$ is the diagonal matrix of charging efficiencies of the units and $\eta^{\mathrm{d}}\in(0, 1]^{S\times S}$ the diagonal matrix of discharging efficiencies, both with nonzero diagonal elements.
	Furthermore, $x^{\mathrm{sd}}\in\R_{\geq 0}^S$ models self-discharge.%

	The simplified storage dynamics of the control-oriented model \eqref{eq:microgridModel:energy:dynamics} can be derived from \eqref{eq:microgridModel:energy:storageLosses} by assuming that $\eta^{\mathrm{c}} = \eta^{\mathrm{d}} = \diag{\One_S}$ and that $x^{\mathrm{sd}} = \Zero_S$.
	Note that \eqref{eq:microgridModel:energy:dynamics} is used in an \ac{mpc} context where the state is sampled at every time instant and used as initial value to predict future states.
	Such a strategy allows to use of less accurate models due to its inherent robustness \cite{rawlings2009model}.
	Furthermore, the error introduced by uncertain renewable infeed and load is assumed to be much larger than the one introduced by \eqref{eq:microgridModel:energy:dynamics} (see \Cref{ass:microgridModel:storageUnits}).
	Therefore, the simplification is of minor importance.
\end{remark}%

\subsection{Transmission network}
\begin{subequations}\label{eq:microgridModel:transmissionNetwork}
The power transmitted over the lines can be derived using \ac{dc} power flow approximations for \ac{ac} grids \cite{HNRR2014, PMVDB2005}.
Thus, the power flowing over the $E$ transmission lines, ${\powerFlow{} = [\powerFlow{1} ~ \cdots ~ \powerFlow{E}]\tran}$, can be derived from the power of the units and the load using the linear relation
\begin{equation}\label{eq:nodePowerToLinePower}{}
	\powerFlow{} = {F} \cdot \begin{bmatrix}\power{}\tran & \disturbance{d} \tran \end{bmatrix} \tran,
\end{equation}
where ${F}\in\R^{E\times(U+D)}$ is a matrix that links the power flowing over the lines with the power provided or consumed by the units and loads.
More information on the derivation of $F$ can be found, e.g., in \cite{HNRR2014,HBR+2018}.
Due to the limited transmission capability of the lines, $\powerFlow{}$ is desired to be bounded by
\begin{equation}\label{eq:network:powerLimit}
	p_{\mathrm{e}}^{\min} \leq \powerFlow{} \leq p_{\mathrm{e}}^{\max}
\end{equation}
with $p_{\mathrm{e}}^{\min}\in\R^E_{\leq 0}$ and $p_{\mathrm{e}}^{\max}\in\R^E_{\geq 0}$.
Additionally, the generated power must equal consumed power at all times, i.e.,
\begin{equation}\label{eq:network:powerEquilibrium}
	\One_{T}\tran \power{t} + \One_{S}\tran \power{s} + \One_{R}\tran \power{r} = \One_{D}\tran \disturbance{d}.
\end{equation}
\end{subequations}

\begin{remark}[]
	There are small \acp{mg} in which all units are directly connected to a single bus, e.g., in \cite{PRG2014,PRG2016}.
	Such \acp{mg} can be modelled in two ways with \eqref{eq:microgridModel:transmissionNetwork}.
	(i)~Discard constraints \eqref{eq:nodePowerToLinePower} and \eqref{eq:network:powerLimit} and only keep \eqref{eq:network:powerEquilibrium}.
	(ii)~Set $F = [\One_{U}\tran ~ -\One_{D}\tran]$.
	Due to \eqref{eq:network:powerEquilibrium}, in this case $p_{e}(k) = 0$ holds for all $k\in\N_0$ and the limits can be chosen as $p_{\mathrm{e}}^{\min} = p_{\mathrm{e}}^{\max} = 0$.%
\end{remark}%

\begin{remark}
	\label[remark]{rem:mgPowerFlowModel}
	For the plant simulation model in \Cref{sec:caseStudy}, we use the nonlinear power flow equations \cite{KBL1994,SHK+2017}.
	Thus, the power provided or consumed by unit or load $i\in\N_{[1, J]}$, where $J = U+D$ is the total number of units and loads, is
	\begin{equation}\label{eq:microgridModel:powerFlowEquations}
		p_{\mathrm{e},i}(k) =
		\hat{v}_i \textstyle\sum\limits_{\subalign{j&=1 \\ j&\neq i}}^{J} \hat{v}_i g_{ij} -
				\hat{v}_j \big(g_{ij} \cos(\theta_{ij}(k)) + b_{ij} \sin(\theta_{ij}(k))\big),
	\end{equation}
	Here, $\hat{v}_i\in \R_{>0}$ is the voltage amplitude at node $i$ and $\theta_{ij}(k) = \theta_i(k) - \theta_j(k)$ is the difference between the phase angles $\theta_i(k)\in\R$ at node $i$ and $\theta_j(k)\in\R$ at node $j\in\N_{[1, J]}$.
	Furthermore, $g_{ij}\in\R_{\geq 0}$ is the conductance and $b_{ij}\in\R$ the susceptance of the line connecting nodes $i$ and $j$.

	Note that \eqref{eq:nodePowerToLinePower} can be derived from \eqref{eq:microgridModel:powerFlowEquations} using \Cref{ass:microgridModel:transmissionLines}, i.e., assuming inductive lines with $g_{ij} = 0$, constant voltage amplitudes $\hat{v}_i$ and small angle differences $\theta_{ij}(k)$ such that $\sin(\theta_{ij}(k)) \approx \theta_{ij}(k)$ can be used for all $i,j\in\N_{[1, J]}$.
	With these assumptions and with $\tilde{b}_{ij} = - \hat{v}_i \hat{v}_j b_{ij}$, \eqref{eq:microgridModel:powerFlowEquations} becomes
	\begin{equation}
		p_{\mathrm{e},i}(k) =
		 \textstyle\sum\limits_{\subalign{j&=1 \\ j&\neq i}}^J
				\tilde{b}_{ij} \theta_{ij}(k).
	\end{equation}
	This equation can now be used to derive matrix $F$ in \eqref{eq:nodePowerToLinePower} (see, e.g., \cite{HNRR2015}).
	Although more accurate convex power flow models exist (see, e.g., \cite{Low2014}), we assume that the error introduced by uncertain renewable infeed and load is much larger than the one introduced by \eqref{eq:nodePowerToLinePower} (see \Cref{ass:microgridModel:transmissionLines}).%
\end{remark}%

\subsection{Overall model}
Having discussed the different components of the islanded \ac{mg}, we can now derive the a control-oriented model of the form \eqref{eq:overall-model} based on \eqref{eq:microgridModel:renewablePower:limits}, \eqref{eq:microgridModel:renewablePower:bigM}--\eqref{eq:microgridModel:storageUnit:powerlimits},
\eqref{eq:microgridModel:powerSharing:storage}--\eqref{eq:microgridModel:transmissionNetwork}.
Here, the auxiliary vector is $q(k) = [\power{}\tran ~ \switchState{r}\tran ~ \frequency(k)]\tran$.
Moreover, ${B = [\Zero_{S \times T} ~ -\samplingTime \Eye_S ~ \Zero_{S\times 2R+1}]}$,
$H_1 = \operatorname{diag}([\One_S\tran ~ -\One_S\tran]\tran)$ and ${h_1 = [(x^{\max})\tran ~ (-x^{\min})\tran]\tran}$.
Furthermore, $H_2$ and $h_2$ in \eqref{eq:overall-model-inequalities} are formed such that they reflect
\eqref{eq:microgridModel:renewablePower:limits},
\eqref{eq:microgridModel:renewablePower:bigM}--\eqref{eq:microgridModel:storageUnit:powerlimits},
\eqref{eq:microgridModel:powerSharing:thermal:bigM} and \eqref{eq:network:powerLimit}.
Additionally, $G$ and $g$ in \eqref{eq:overall-model-algebraic} are formed such that they reflect \eqref{eq:microgridModel:powerSharing:storage}, \eqref{eq:nodePowerToLinePower} and \eqref{eq:network:powerEquilibrium}.%


\section{Uncertainty model}
\label{sec:forecastAndScenarioTree}

This section focuses on the representation of uncertain load and renewable infeed.
First the generation of a collection of forecast scenarios is discussed.
Then scenario-trees are illustrated and the model of an \ac{mg} with uncertain load and renewable generation is derived.

\subsection{Representation of uncertainty by collections of scenarios}

\begin{figure}
	\centering
	\includegraphics{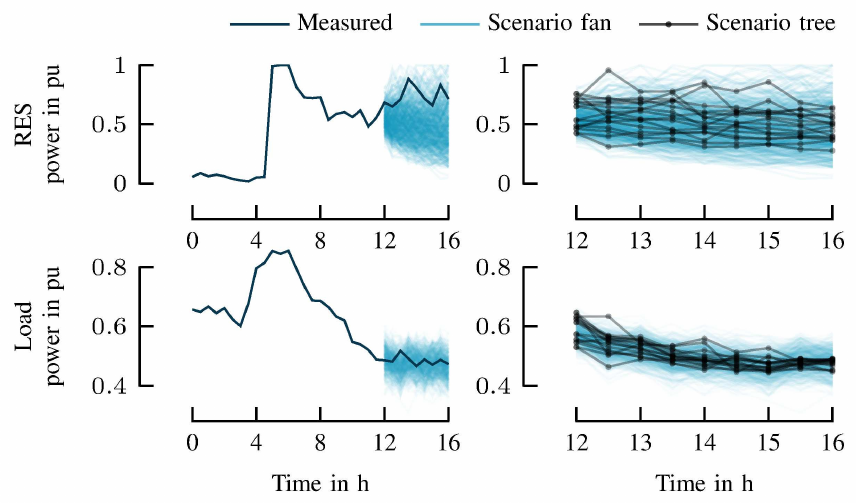}
	\caption{Wind and load power forecast with seasonal \ac{arima} models from \cite{HSB+2015} over a horizon of \unit[4]{h}.
	The collection of independent forecast scenarios was generated using Monte Carlo simulations with 500 seeds.
	The maximum number of children per node was chosen as 8 for the root node, 2 for the nodes of stage 1 and 1 for all other stages.}
	\label{fig:windAndLoadFcMCScenariosReducedScenarios}
\end{figure}

To obtain a representative probability distribution of load and renewable infeed for the controller, a sampling based Monte Carlo forecast was chosen.
Here, random samples that follow the error distribution obtained from the training of the forecast model are drawn and applied to the forecast to generate a collection of independent scenarios where every scenario has the same probability.
Thus, for a high number of independent forecast scenarios the probability distribution of the forecast is approximated.
To generate scenarios of load and available renewable power, the seasonal \ac{arima} models from \cite{HSB+2015} were used (see~\Cref{fig:windAndLoadFcMCScenariosReducedScenarios}).
For more information on time-series based forecasting, the reader is referred to \cite{BJR2013}.

To achieve a sufficiently accurate approximation of the forecast probability distribution, a high number of independent forecast scenarios is desired.
This leads to an undesired high computational complexity in finding a suitable control action as a high number of scenarios often leads to a high number of decision variables.
To satisfy both needs sufficiently, we generate scenario trees which serve as a more compact representation of the probability distribution.

\subsection{From data to scenario trees}

Scenario trees can be constructed from collections of forecast scenarios.
They can be obtained using methodologies such as~\cite{Pflug2015scen} or {scenario reduction} (see, e.g., \cite{HR2003,HeiRom09}).
There exist several other scenario generation algorithms, e.g., clustering-based \cite{Latorre20071339,Chen:2014:KST:2787349.2787351} as well as simulation and optimisation-based approaches~\cite{Beraldi20102322,Gulpinar20041291}.
An example of such a collection of independent forecast scenarios and the corresponding scenario tree for a forecast of load and available wind power is shown in \Cref{fig:windAndLoadFcMCScenariosReducedScenarios} on the right.
This tree was generated using fast forward selection as described by \cite[Algorithm~5]{OPG+2015} which is a modification of \cite[Algorithm~2.4]{HR2003}.

\subsection{Representation of uncertainty using scenario trees}
\label{sec:scenario_trees}

\begin{figure}
	\centering
	\includegraphics{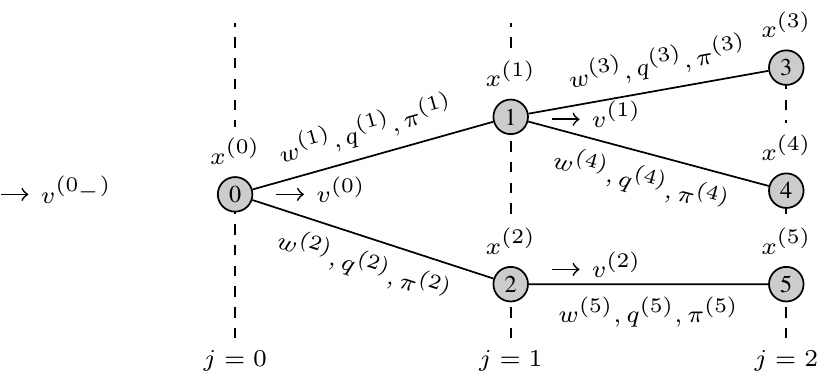}
	\caption{Probability tree with $N=2$ and $\mu=6$ nodes.
		The nodes of the stages are $\nodes(0) = \{0\}$ as well as $\nodes(1) = \{1, 2\}$ and $\nodes(2) = \{3, 4, 5\}$.
		Note that $\ancestor(1)=0$, $\ancestor(3)=1$ and $\ancestor(5)=2$.
		Also $\child(0)=\{1,2\}$, $\child(1)=\{3,4\}$ and $\child(2)=\{5\}$.
        The scenarios of the tree are the sequences $(0,1,3)$, $(0,1,4)$ and $(0,2,5)$.
        }%
	\label{fig:probabilitytreeUnsymmetricVariables}
\end{figure}

A scenario tree is a representation of the uncertain evolution of a discrete-time finite-valued random process as illustrated in \Cref{fig:probabilitytreeUnsymmetricVariables}.
A tree is a collection of $\mu\in\N$ nodes partitioned into stages $j \in \N_{[0, N]}$ and indexed with a unique identifier $i\in\N_{[0,\mu-1]}$.
Each node is associated with a possible value of the state of the process at a future time instant starting from an initial node $i=0$ at stage $j=0$ which is called the root node of the tree.
The set of nodes at stage $j$ is denoted by $\nodes(j)\subseteq \N_{[0,\mu-1]}$.
Conversely, the stage in which a node $i$ resides is denoted by $\stage(i)\in\N_{[0,N]}$.
The nodes at stage $j=N$ are called leaf nodes.
All non-leaf nodes $i$ at a stage $j$ possess a set of {child nodes} which are in stage $j+1$ and are connected to $i$; these are denoted by $\child(i)\subseteq \nodes(j+1)$.
Likewise, every node $i\neq 0$ is reachable from a single ancestor node, which resides in the previous stage and is denoted by $\ancestor(i)\in \nodes(\stage(i)-1)$.
A scenario is a sequence of nodes $(s_0, \ldots, s_N)$ such that $s_{N}\in\nodes(N)$ and $\ancestor(s_a) = s_{a-1}$, $a \in\N_{[1,N]}$.
Scenarios are uniquely identified by leaf nodes.
The probability of visiting node $i\in\N_{[0,\mu-1]}$ is denoted by $\pi^{(i)}>0$.
That said, at stage $j$, the set $\nodes(j)$ is a probability space with $\sum_{i\in \nodes(j)}\pi^{(i)}=1$.

For the tree shown in \Cref{fig:probabilitytreeUnsymmetricVariables}, the relation of the different variables at $j=0$ and $j=1$ is illustrated in \Cref{fig:probabilitytreeVariables}.
As shown here, we make a decision $v^{(0)}$ at stage $j=0$ without knowing which disturbance $w^{(1)}$ or $w^{(2)}$ will occur during the time between $j=0$ and $j=1$.
Depending on the disturbance $w$, different values of the auxiliary vector $q$ will occur.
These can be calculated using the function $f_q(\cdot, \cdot)$ derived from \eqref{eq:overall-model}.
Thus, the choice of $v^{(0)}$ accounts for $q^{(1)} = f_q(v^{(0)}, w^{(1)})$ and $q^{(2)} = f_q(v^{(0)}, w^{(2)})$ without knowing which one occurs.
Similarly, following \eqref{eq:overall-model-dynamics} the state $x$ at time instant $j=1$ is a function $f_x(\cdot, \cdot)$ of the state at time instant $j=0$ and the auxiliary vector that is present between $j=0$ and $j=1$.
Consequently, different values $q^{(1)}$ and $q^{(2)}$ lead to different states $x^{(1)} = f_x(x^{(0)}, q^{(1)}) = f_x(x^{(0)}, f_q(v^{(0)}, w^{(1)}))$ and $x^{(2)} = f_x(x^{(0)}, q^{(2)}) = f_x(x^{(0)}, f_q(v^{(0)}, w^{(2)}))$.

\begin{figure}
	\centering
	\includegraphics{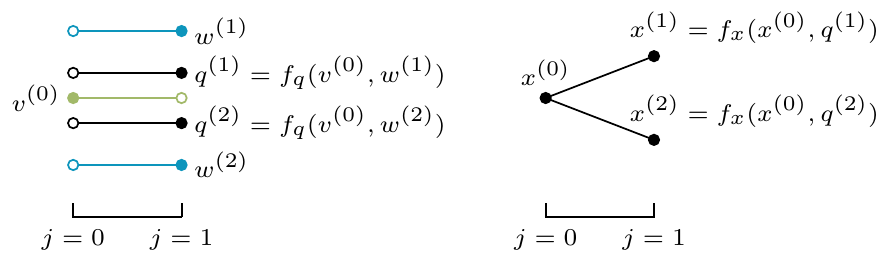}
	\caption{%
		Relation of optimisation variables on scenario tree for power (left) and energy (right).
		Note that the power values are piecewise constant.
		The power setpoints $v^{(0)}$ are represented by the green line,
		the uncertain inputs $w^{(1)}$, $w^{(2)}$ by blue lines
		and
		the auxiliary variables $q^{(1)}$, $q^{(2)}$ by black lines.
		Functions $f_q$ and $f_x$ are given implicitly via~\eqref{eq:overall-model}.
	}%
	\label{fig:probabilitytreeVariables}%
\end{figure}

In summary, given an initial measured system state $x^{(0)}$ at stage $j=0$ and a forecast in the form of $w^{(1)}$ or $w^{(2)}$ we make a decision $v^{(0)}$ using this information.
Similarly, at every stage $j\in\N_{[0,N-1]}$, we make decisions $v^{(i)}$, $i\in\nodes(j)$, using the information that is available up to that stage and using the forecast values $w^{(i_+)}$ for $i_+\in\child(i)$.
In other words, $v$ is decided using {causal} control laws.
This is indicated in \Cref{fig:probabilitytreeUnsymmetricVariables} by the positioning of $v^{(i)}$ at the nodes of the tree instead of at its edges.
Thus, across the nodes of the scenario trees, the control-oriented model \eqref{eq:overall-model} for all nodes $i_+\in\N_{[1,\mu-1]}$ with $i = \ancestor(i_+)$ become
\begin{subequations}\label{eq:forecastAndScenarioTree:constraints}
\begin{align}
	x^{(i)} + {B}q^{(i_+)} - x^{(i_+)} &= 0, \label{eq:forecastAndScenarioTree:constraints:dynamics}\\
 	H_1 \cdot x^{(i_+)}{}  &\leq h_1, \label{eq:forecastAndScenarioTree:constraints:dynamicsLimit} \\
 	H_2 \cdot \begin{bmatrix} v^{(i)}{}\tran & q^{(i_+)}{}\tran & w^{(i_+)}{}\tran \end{bmatrix}\tran &\leq h_2,\label{eq:forecastAndScenarioTree:constraints:inequalityConstraints}\\
	G \cdot \begin{bmatrix} v^{(i)}{}\tran & q^{(i_+)}{}\tran & w^{(i_+)}{}\tran \end{bmatrix}\tran &= g. \label{eq:forecastAndScenarioTree:equatlityContstraints}
\end{align}
\end{subequations}


\section{Operating costs}
\label{sec:operatingCosts}

\newcommand{\vi}{\ensuremath{v^{(i)}}}
\newcommand{\ziPlus}{\ensuremath{q^{(i_+)}}}
\newcommand{\piPlus}[1]{\ensuremath{p^{(i_+)}_{\mathrm{#1}}}}

In this section we derive an operating cost function for an \ac{mg}
which reflects the main objectives:
(i)~economic operation,
(ii)~low number of switching actions,
(iii)~high use of \ac{res} and
(iv)~desired state of storage units.
We use cost functions that are motivated by~\cite{HSB+2015}.
Our presentation will hinge on the scenario tree structure
introduced in \Cref{sec:scenario_trees}, i.e., objectives
will be defined at the nodes of a scenario tree.

The objective at node $i_+\in\N_{[1, \mu-1]}$ with $i = \ancestor(i_+)$ and $i_- = \ancestor(i)$ is composed
of $\ell_{\mathrm{o}}(v^{(i)}, v^{(i_-)}, \ziPlus)\in\R_{\geq 0}$ that reflects items (i)--(iii),
and $\ell_{\mathrm{s}}(x^{(i_+)}) \in \R_{\geq 0}$ that reflects (iv).
A discount factor ${\gamma\in (0,1)}$ is used to emphasise decisions in the near future.
Thus, the cost associated with node $i_+$ is
\begin{multline}
	\ell(x^{(i_+)}, v^{(i)}, v^{(i_-)}, \ziPlus) = \\ \gamma^{\stage(i_+)} \big(
	\ell_{\mathrm{o}}(v^{(i)}, v^{(i_-)}, \ziPlus) 
	+ \ell_{\mathrm{s}}(x^{(i_+)}) \big) .
\end{multline}

\begin{remark}
	Note that the cost associated with node $i_+$, $\ell_{\mathrm{o}}(v^{(i)}, v^{(i_-)}, \ziPlus)$, also depends on nodes $i = \ancestor(i_+)$ and $i_-=\ancestor(i)$.
	The reason for this is that $q^{(i_+)}$ is a function of the input $v^{(i)}$ (see \Cref{sec:scenario_trees}).
	Therefore, it is required to use $q^{(i_+)}$ and $v^{(i)}$ with $i = \ancestor(i_+)$ in the same cost function.
	As the cost that is caused by $v^{(i)}$ partly depends on the input at the past time instant, e.g., in case of switching penalties, it is further required to include $v^{(i_-)}$.%
\end{remark}

The economically motivated cost includes
(i)~operating costs of conventional units, $\ell_{\mathrm{t}}^{\mathrm{rt}} (\vi, \ziPlus)\in\R_{\geq 0}$,
(ii)~costs for switching conventional units on or off, $\ell_{\mathrm{t}}^{\mathrm{sw}}(v^{(i)}, v^{(i_-)})\in\R_{\geq 0}$, and
(iii)~costs incurred by low utilisation of renewable sources, $\ell_\mathrm{r}(\ziPlus)\in\R_{\geq 0}$,
i.e.,
\begin{multline}
	\label{eq:operatingCosts:economic}
	\ell_{\mathrm{o}}(v^{(i)}, v^{(i_-)}, \ziPlus) = \ell_{\mathrm{t}}^{\mathrm{rt}} (\vi, \ziPlus) +
	\\
	\ell_{\mathrm{t}}^{\mathrm{sw}}(v^{(i)}, v^{(i_-)}) +
	\ell_\mathrm{r}(\ziPlus).
\end{multline}

\begin{subequations}
More precisely, following \cite{JMK2012}, the operating cost of the conventional units is modelled as
\begin{equation}
	\ell_{\mathrm{t}}^{\mathrm{rt}} (\vi, \ziPlus)
	= c_{\mathrm{t}}\tran \delta_{\mathrm{t}}^{(i)} + {c_{\mathrm{t}}^\prime}\tran \piPlus{t} + \|\diag{c_{\mathrm{t}}^{\prime\prime}} \piPlus{t} \|_2^2 ,
\end{equation}
with weights $c_{\mathrm{t}}\in\R^T_{>0}$,
$c_{\mathrm{t}}^{\prime}\in\R^T_{>0}$ and $c_{\mathrm{t}}^{\prime\prime}\in\R^T_{>0}$ and using the square of the Euclidean norm $\| \cdot \|_2^2$.

Costs for switching the conventional units on or off from node $i_- = \ancestor(i)$ to node $i$ are modelled by
\begin{equation}
	\ell_{\mathrm{t}}^{\mathrm{sw}}(v^{(i)}, v^{(i_-)}) =  \|\diag{c_{\mathrm{t}}^{\mathrm{sw}}} (\delta_\mathrm{t}^{(i_-)} - \delta_\mathrm{t}^{(i)} ) \|_2^2
\end{equation}
with weight $c_{\mathrm{t}}^{\mathrm{sw}}\in\R^T_{>0}$.
For the root node, $\delta_\mathrm{t}^{(0_-)}$ denotes the initial switch state of the conventional units, $\delta_{\mathrm{t},k-1}$.

It is desired to maximise renewable infeed.
This can be included in the stage cost as a penalty if the renewable infeed is less than the nominal value $p_{\mathrm{r}}^{\Max}$, i.e., with $c_{\mathrm{r}}\in\R^R_{>0}$,
\begin{equation}
	 \ell_\mathrm{r}(\ziPlus) = \|\diag{c_{\mathrm{r}}} (p_{\mathrm{r}}^{\max} - \piPlus{r})\|_2^2.
\end{equation}
\end{subequations}

Very high or very low values of $x(k)$, can increase the ageing of batteries \cite{VNW+2005}.
To reduce the occurrence of such values, we introduce the interval of desired values of $x(k)$, $[\tilde{x}^\Min, \tilde{x}^\Max] \subseteq [{x}^\Min, {x}^\Max]$.
To enforce $x(k) \in [\tilde{x}^\Min, \tilde{x}^\Max]$ we define $\ell_{\mathrm{s}}(x^{(i)})$ with $c_\mathrm{s}\in\R^S_{>0}$ as
\begin{equation}\label{eq:softConstraints:cost}
	\ell_{\mathrm{s}}(x^{(i)}) = c_\mathrm{s}\tran (\max(\tilde{x}^\Min - x^{(i)}, \Zero_S) - \min(\tilde{x}^\Max - x^{(i)} , \Zero_S)).
\end{equation}

\begin{subequations}\label{eq:costVectorAtStageJPlusOne}
For every node $i_+\in\N_{[1, \mu-1]}$ with $i_+ \in \child(i)$ as well as $i_- = \ancestor(i)$ we define the cost variable
\begin{equation}
	Z^{(i_+)} = \ell(x^{(i_+)}, v^{(i)}, v^{(i_-)}, \ziPlus).
\end{equation}
Thus, for stage $j \in \N_{[1,N]}$, the vector
\begin{equation}\label{eq:Z-stage-j}
	Z_{j} = [Z^{(i_+)}]_{i_+ \in \nodes(j)}
\end{equation}
is associated with a random variable on the probability space $\nodes(j)$.
Note that $Z_{j}\in\R^{|\nodes(j)|}$ is a function of states $x^{(i_+)}$, inputs $v^{(i)}$, $v^{(i_-)}$ and auxiliary variables $\ziPlus$ for all $i_+\in\nodes(j)$, $i = \ancestor(i_+)$ and $i_- = \ancestor(i)$.
Using \eqref{eq:Z-stage-j}, we can describe the multi-stage cost by the sequence $(Z_1, \ldots, Z_{N})$.
\end{subequations}

\section{Measuring risk}
\label{sec:measuringRisk}

In this section we introduce the notion of risk
measures and provide a few examples thereof.
Furthermore, we will discuss one specific risk measure: the \acf{avar}.

\subsection{Introduction to risk measures}

\newcommand{\probSimplex}{\ensuremath{\mathbb{D}}}
\newcommand{\admissSet}{\ensuremath{\mathbb{A}}}

Let $\Omega = \{\omega_1, \ldots, \omega_K\}$ be a sample space, whose elements $\omega_i$ have probabilities $\pi_i > 0$ for $i \in \N_{[1, K]}$.
The probabilities can be collected in a vector $\pi=[\pi_1 ~\cdots ~\pi_K]\tran$.
This vector is an element of the probability simplex, i.e., the set ${\probSimplex=\{\pi\in\R^K{}\mid{}\pi_i \geq 0,\ \sum_{i=1}^K \pi_i=1\}}$.
A random variable on $\Omega$ is a function $\bar{Z}:\Omega\to\R$ with $\bar{Z}(\omega_i)=Z_i$.
All values of $\bar{Z}$ can be collected in a vector $Z=[Z_1 ~ \cdots ~ Z_K]\tran\in\R^K$.
Note that in our case, these values are given by \eqref{eq:costVectorAtStageJPlusOne} and represent the operation cost at the nodes of the scenario tree.

A risk measure on $\Omega$ is a mapping $\rho:\R^K\to\R$ that, roughly speaking, quantifies the significance of extreme events.
A well-known, yet trivial, risk measure is the expectation operator
$\E_\pi(Z) = \sum_{i=1}^{K}\pi_i Z_i$,
which is often referred to as a risk-neutral measure as it carries no deviation information.
Another example of a risk measure is the maximum operator
$\max(Z) = \max\{Z_i{}\mid{} i{=}1,{\ldots},K\}$ which quantifies the
worst-case value, or realisation, of $Z$.
It can be written as
\begin{equation}\label{eq:max-dual-representation}
  \max(Z)  = \max_{\pi^{\prime}\in\probSimplex}\E_{\pi^{\prime}}(Z),
\end{equation}
where the maximum is taken with respect to all probability
vectors $\pi^{\prime}\in\probSimplex$.
Therefore, it can be interpreted as the worst-case expectation
over all possible probability distributions.
Following the notation in \eqref{eq:max-dual-representation}, the expectation operator is
\begin{equation}\label{eq:expectation-dual-representation}
  \E_{\pi}(Z) = \max_{\pi^{\prime}\in \{\pi\}} \E_{\pi^{\prime}}(Z),
\end{equation}

In this work we focus on coherent%
\footnote{%
Let $\bar{Z}$, $\bar{Z}^\prime$ be two random variables on $\Omega$ and ${Z}$, ${Z}^\prime$ the corresponding vectors.
Then, $\rho:\R^K\to\R$ is a coherent risk measure (see \cite[Def.~6.4.]{SDR2014}) if it is
(i)~convex, i.e., $\rho(\lambda Z + (1-\lambda) {Z}^\prime) \leq \lambda \rho(Z) + (1-\lambda) \rho({Z}^\prime)$ for all $\lambda\in[0,1]$,
(ii)~monotone, i.e., $\rho(Z) \leq \rho({Z}^\prime)$ whenever $Z \leq {Z}^\prime$,
(iii)~translation equi-variant, i.e., $\rho(c \One_K + Z) = c + \rho(Z)$ for all $c\in\R$, and
(iv)~positive homogeneous, i.e., $\rho(\alpha Z) = \alpha \rho(Z)$ for all $\alpha\in\R_{\geq 0}$.

}
risk measures as they quantify risk in a natural and intuitive way.
Also, they allow for a computationally tractable reformulation of risk-averse problems which renders them suitable for \ac{mpc} applications.

It is shown in~\cite[Thm.~6.5]{SDR2014} that all coherent risk measures can be written in a form reminiscent of \eqref{eq:max-dual-representation} and~\eqref{eq:expectation-dual-representation} as
\begin{equation}\label{eq:risk-measure-dual-representation}
    \rho(Z) = \max_{\pi^{\prime}\in \admissSet}\E_{\pi^{\prime}}(Z),
\end{equation}
where $\admissSet\subseteq \probSimplex$ is a closed convex set which contains $\pi$.
Every mapping of the form~\eqref{eq:risk-measure-dual-representation}, where the so-called ambiguity set $\admissSet$ of $\rho$ is a closed convex set which contains $\pi$, is coherent.
One interpretation of \eqref{eq:risk-measure-dual-representation} is that we take the worst-case expectation of $Z$ with respect to $\pi^{\prime}\in\admissSet$, that is, with respect to an inexactly known distribution \cite{van2016distributionally}.
The expectation and maximum operators are two extreme cases of coherent risk measures.
The ambiguity set of $\max(Z)$ is the largest possible set, i.e., $\admissSet = \probSimplex$.
The ambiguity set of $\E(Z)$ is the smallest possible set, i.e., $\admissSet = \{\pi\}$.
Other risk measures can be constructed by taking ambiguity sets of intermediate size to cope with uncertain knowledge of a probability distribution.

Risk measures, whose ambiguity set is a polytope, are called polytopic.
Given the extreme points of the ambiguity set, its vertices $\pi_{1}, \ldots, \pi_{L}$, they assume the convenient representation
\begin{equation}
 \label{eq:polytopic_risk_measure}
 \rho(Z) = \max_{l\in\N_{[1, L]}} \E_{\pi_{l}}(Z).
\end{equation}

\subsection{Average value-at-risk}
\label{sec:avar}

\begin{figure}
  \centering
  \includegraphics{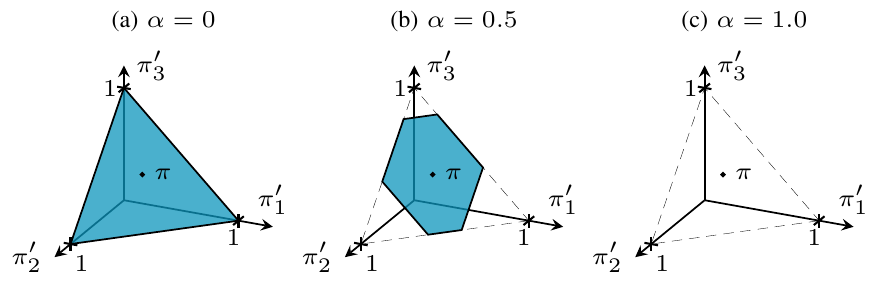}%
  \caption{Ambiguity sets $\admissSet_\alpha$ over a probability space with $K=3$ and $\pi=[0.3 ~ 0.3 ~ 0.4]^{\scriptscriptstyle\top}$.
           The ambiguity set $\admissSet_0$ is the whole probability simplex,
           and $\admissSet_{1}$ the point $\{\pi\}$.
           Naturally, $\admissSet_{\alpha_1}\subseteq \admissSet_{\alpha_2}$ for $\alpha_1 \geq \alpha_2$.}
  \label{fig:avar-admissibility}
\end{figure}

\begin{subequations}\label{eq:avar:admissibilitySet}
A commonly used risk measure is the \acl{avar}, which is given in the form of \eqref{eq:risk-measure-dual-representation} as
\begin{equation}
    \rho(Z) = \AVaR (Z) =\max_{\pi^{\prime}\in \admissSet_{\alpha}}\E_{\pi^{\prime}}(Z)
\end{equation}
with the ambiguity set
\begin{equation}
    \admissSet_{\alpha} = \begin{cases}
        \{\pi^{\prime}\in\probSimplex {}\mid{} \pi^{\prime} \leq \tfrac{1}{\alpha}\pi\},&\text{ if } \alpha \in (0,1],
        \\
        \probSimplex,&\text{ if } \alpha=0.
        \end{cases}
\end{equation}
for $\alpha\in\R_{[0,1]}$.
Clearly, \ac{avar} is a polytopic risk measure since $\admissSet_{\alpha}$ is a polytope.
As shown in \Cref{fig:avar-admissibility}, $\admissSet_\alpha$ can be modified by varying $\alpha$.
This includes the extreme cases $\alpha=1$, where $\admissSet_1 = \{\pi\}$ and $\alpha = 0$, where $\admissSet_0 = \probSimplex$.
\end{subequations}
Using convex duality arguments and the additional free variable ${t\in\R}$, \eqref{eq:avar:admissibilitySet} can be transformed into~\cite[Ex.~6.19]{SDR2014}
\begin{equation}\label{eq:avar:minEssmax}
 \AVaR (Z) {=} \begin{cases}
                \min
        \big(t + \textstyle \\
        \quad \E_{\pi}\big(  \max\big(\frac{Z - t \One_{K}}{\alpha}, \Zero_{K} \big)\big)\big),
			    & \text{for } \alpha \in (0,1]
		\\
                \max(Z), &{} \text{for } \alpha = 0.
              \end{cases}
\end{equation}

We will now give an equivalent representation of the average value-at-risk
which will be useful in \Cref{sec:reformulation}
as it will facilitate the solution of risk-averse optimal control problems.

\begin{proposition}\label[proposition]{prop:avar}
 The \acl{avar} at level $\alpha\in[0,1]$ is given by
 \begin{align}
 \AVaR (Z)
        =
        \min_{
    \subalign{
    \xi {}\geq{}& \Zero_K \\
    \alpha \xi {}\geq{}& Z {}-{} t\One_K \hspace{-2em}}
    }
    (t + \E_{\pi}(\xi)). \label{eq:avarEpigraphRelaxation}
\end{align}

\begin{proof}
 Let $\alpha \in (0,1]$. Using the epigraphical relaxation (see \cite[Sec.~3.1.7 \& 4.1]{BV2004}) of $\max({}\cdot{}, 0)$,
 we have that
 \begin{equation}
    \max(y, \Zero_{K}) = \min_{\subalign{
      \xi &\geq \Zero_{K} \\
      \xi &\geq y
    }}
    \xi,
    \label{eq:max_epi_relaxation}
 \end{equation}
 for all $y\in\R^K$,
 where $\xi\in\R^{K}$ is a slack variable.
 Therefore,
 \begin{align*}
  \AVaR(Z)
  {}={}&
  \min
  \left(
    t {}+{}
    \E_{\pi} \big( \max\big(\tfrac{Z - t \One_{K}}{\alpha}, \Zero_{K} \big)\big)
    \right)
  \\
  {}={}& \min
    \Big(
      t + \E_{\pi} \big(
      \min_{\subalign{
        \xi &\geq \Zero_{K} \\
        \xi &\geq \tfrac{Z - t \One_{K}}{\alpha} \hspace{-1.7em}
      }}
      \hspace{1.0em} \xi
      \big)
    \Big)
  \\
  {}={}& \min
    \Big(
      t + \E_{\pi} \big(
      \hspace{0.5em}
      \min_{\subalign{
        \xi &\geq \Zero_{K} \\
        \hspace{-0.5em} \alpha \xi &\geq Z - t \One_{K} \hspace{-1.7em}
      }}
      \hspace{1.0em} \xi
      \big)
    \Big),
 \end{align*}
 where the second equation is by virtue of \eqref{eq:max_epi_relaxation}.
 Using \cite[Prop. 6.60]{SDR2014}, we interchange the expectation operator, $\E_{\pi}$ and the minimum to arrive at \eqref{eq:avarEpigraphRelaxation}.

 The right hand side of \eqref{eq:avarEpigraphRelaxation} is well defined
 for $\alpha=0$, i.e.,
 \begin{align*}
  \mathrm{AV@R}_{0} (Z)
  &=
        \min_{
    \subalign{
    \xi {}\geq{}& \Zero_K \\
    \hspace{-0.1em} t\One_K \geq{}& Z
    }}
    (t + \E_{\pi}(\xi)) \\
  &=
    \min_{\hspace{-0.1em} t\One_K \geq Z} t {}+{}
    \min_{\xi {}\geq{} \Zero_{K}} \E_{\pi}(\xi)
  {}={} \min_{t\One_K \geq Z} t
  {}={} \max (Z).
 \end{align*}
Therefore, \eqref{eq:avarEpigraphRelaxation} holds for all $\alpha \in [0,1]$.
\end{proof}%
\end{proposition}%

Having discussed risk measures, we will now use them to construct multistage risk-averse \ac{mpc} problems.


\section{Risk-Averse \ac{mpc}}
\label{sec:riskAverseMpc}
In this section we will first introduce conditional risk mappings on scenario trees.
Then we will construct multistage risk-averse \ac{mpc} problems and reformulate them as \acp{miqcqp}.

\subsection{Conditional risk mappings on scenario trees}
\label{sec:conditional_risk_mappings}

A conditional risk mapping on a scenario tree is a generalisation of the notion of conditional expectation
which is the expectation of a random cost $Z_{j+1}$ at stage $j+1$ given
all information we can surmise up to stage $j$.
Roughly speaking, a conditional risk mapping at a non-leaf node $i$ of the tree
returns the risk of the cost of the children of $i$.
Conditional risk mappings can be constructed as follows.

For every stage $j\in\N_{[0,N-1]}$, the set $\nodes(j)$ is a probability
space whose elements $i \in \nodes(j)$ have probability $\pi^{(i)}$.
Naturally, we can define real-valued random variables with corresponding vectors $Z_{j}\in\R^{|\nodes(j)|}$ on that space.
Likewise, the set $\nodes(j+1)$ is also a probability space.
A conditional risk mapping at stage $j$ is a mapping
\begin{equation}\label{eq:risk-mapping-structure}
  \rho_j: \R^{|\nodes(j+1)|} \to \R^{|\nodes(j)|},
\end{equation}
which is constructed as we explain hereafter.

\begin{figure}[t]
  \centering
  \includegraphics{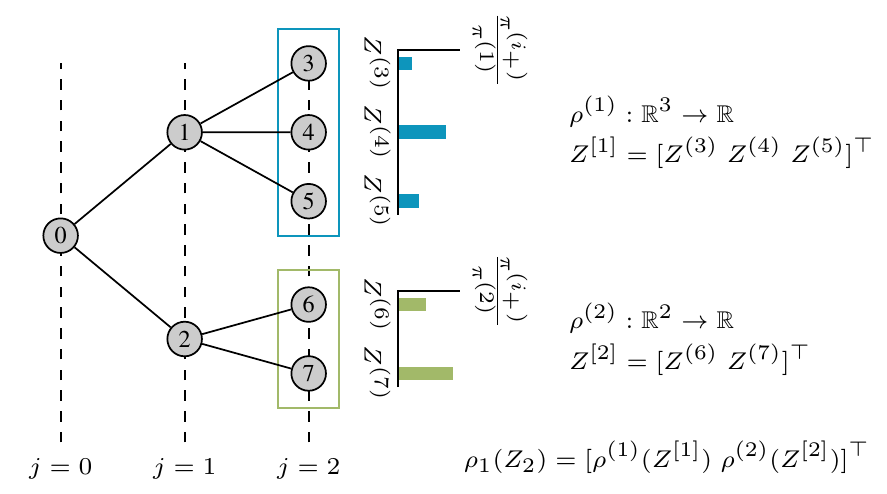}%
  \caption{%
  	Example of conditional risk mapping conditioned at stage $j=1$.
    Following \eqref{eq:risk-mapping-structure}, $\rho_1$ is a conditional risk mapping $\rho_1:\R^5 \to \R^2$.
    Note here that $Z_2=[Z^{(3)} ~ \cdots ~ Z^{(7)}]^{\scriptscriptstyle\top}$, is decomposed into $Z^{[1]}$ and $Z^{[2]}$.
	}
    \label{fig:probabilitytreeRiskMapping}
\end{figure}

For all $i\in\nodes(j)$, the sets $\child(i)\subseteq\nodes(j+1)$ are disjoint
and define a partition over $\nodes(j+1)$, i.e.,
\(
 \textstyle \nodes(j+1) {}={} \bigcup_{i\in\nodes(j)}\child(i).
\)
Given that node $i$ is visited, $i_+\in\child(i)$ occurs with probability $\nicefrac{\pi^{(i_+)}}{\pi^{(i)}}$.
This makes $\child(i)$ into a probability space whereon we can construct
random variables with vectors
\(
	Z^{[i]}=\big[Z^{(i_+)}\big]_{i_+\in\child(i)}.
\)
In particular, $Z_{j+1}$ on the probability space $\nodes(j+1)$ is decomposed
into vectors $Z^{[i]}$ on $\child(i)$, $i\in\nodes(j)$.

Using a coherent risk measure $\rho^{(i)}:\R^{|\child(i)|}\to\R$,
we can compute a risk $\rho^{(i)}(Z^{[i]})$ for every vector $Z^{[i]}$.
Thus, we define the conditional risk mapping $\rho_j$ at stage $j$ as
\begin{equation}\label{eq:risk-mapping:definition}
	\rho_j(Z_{j+1}) {}={}
	\big[
		\rho^{(i)}(Z^{[i]})
	\big]_{i\in\nodes(j)}.
\end{equation}
In words, $\rho_j$ maps the probability distribution at stage $j+1$ into a vector whose $i$th element denotes the risk that will incur if one is at node $i\in\nodes{(j)}$.
For a simple scenario tree, this is illustrated in \Cref{fig:probabilitytreeRiskMapping}.

\begin{remark}
For the case where $\rho^{(i)}(Z^{[i]})$ is $\AVaR(Z^{[i]})$,
according to \cref{prop:avar}, the risk of $Z^{[i]}$ is
\begin{equation}\label{eq:avarEpigraphRelaxation:node}
	\rho^{(i)}(Z^{[i]})
 	= \min_{
		\subalign{
		\xi^{[i]}	&\geq{} \Zero_{|\child(i)|} \hspace{-3.5em} \\
		\alpha \xi^{[i]} &\geq{} Z^{[i]} {}-{} t^{(i)} \One_{|\child(i)|}\hspace{-5.5em}
		}}
			\big(t^{(i)} + \E_{\pi^{[i]}}(\xi^{[i]})\big)
\end{equation}
with $t^{(i)}\in\R$
and $\xi^{[i]} = [\xi^{(i_+)}]_{i_+\in\child(i)}$, $\xi^{(i_+)}\in\R$
for $i\in\N_{[0,\mu-1]} \setminus \nodes(N)$.
As we consider the probability space $\child(i)$, we are interested in the probability of visiting $i_+\in\child(i)$ given that we are at node $i$.
Therefore, the probabilities are $\pi^{[i]} = [\nicefrac{\pi^{(i_+)}}{\pi^{(i)}}]_{i_+\in\child(i)}$ and \eqref{eq:avarEpigraphRelaxation:node} is equivalent to
\begin{equation}\label{eq:avarEpigraphRelaxation:node:long}
   	\rho^{(i)}(Z^{[i]})
 	=
 	\min_{
		\subalign{
		\xi^{[i]}	&\geq{} \Zero_{|\child(i)|} \hspace{-3.5em} \\
		\alpha \xi^{[i]} &\geq{} Z^{[i]} {}-{} t^{(i)} \One_{|\child(i)|} \hspace{-5.5em}
		}}
			\big(t^{(i)} + \textstyle \sum_{i_+\in\child(i)} \tfrac{\pi^{(i_+)}}{\pi^{(i)}} \xi^{(i_+)} \big).
\end{equation}
\end{remark}

Having discussed conditional risk mappings on scenario trees, we can now use them to construct a multistage risk~measure based on the vector $Z = [Z_1\tran ~ \cdots ~ Z_N\tran]\tran$ that is associated with the multistage random variable.

\subsection{Risk-averse optimal control}

Let $Z_j\in\R^{|\nodes(j)|}$ be the cost in \eqref{eq:costVectorAtStageJPlusOne}.
For the sequence $(Z_1,\ldots, Z_N)$
and given a sequence of conditional risk mappings $\rho_j$ of the form~\eqref{eq:risk-mapping:definition},
the nested multistage risk measure $\varrho_N:\R^{|\nodes(1)|}\times\cdots\times\R^{|\nodes(N)|}\to\R$ is \cite{HSBP2017,CP2014,SDR2014}%
\begin{multline}\label{eq:varRho}
	\varrho_N(Z_1, \ldots, Z_N) = \\
	\rho_0 \left(Z_1+ \rho_1\left(Z_2 + \ldots + \rho_{N-1}(Z_N)\right) \ldots \right).
\end{multline}
Nested multistage risk measures possess favourable properties which render them
suitable for optimal control formulations.
The most important properties are that they
(i)~are suitable for multistage formulations as they measure how risk propagates over time,
(ii)~are coherent risk measures over the space $\R^{|\nodes(1)|} \times \cdots \times \R^{|\nodes(N)|}$~\cite[Sec.~6.8]{SDR2014},
(iii)~give rise to optimal control problems which are amenable to dynamic
       programming formulations~\cite{Shapiro2012},
(iv)~allow for \ac{mpc} formulations with closed-loop stability guarantees~\cite{CP2014,HSBP2017}.

The definition of $\varrho_N$ in \eqref{eq:varRho} allows to formulate the following optimisation problem with decision variables
${\boldsymbol{v} = [v^{(i)}]_{i\in\N_{[0, \mu-1]}}}$,
$\boldsymbol{x} = [x^{(i)}]_{i\in\N_{[0, \mu-1]}}$,
$\boldsymbol{q} = [q^{(i)}]_{i\in\N_{[0, \mu-1]}}$.

\begin{problem}
	\label{prob:riskAverseMpc:nestedConditionalRiskMappings}
Solve the optimal control problem
\[
	\Minimize_{\boldsymbol{v}, \boldsymbol{x}, \boldsymbol{q}}\ \varrho_N(Z_1,\ldots, Z_N)
\]
$\subjectto$
\begin{align*}
	& \text{constraints \eqref{eq:forecastAndScenarioTree:constraints} and given initial conditions } x^{(0)}, \delta_{\mathrm{t}}^{(0_-)} \\
	& \forall i_+\in\N_{[1,\mu-1]} \text{ and }  i = \ancestor(i_+).
\end{align*}
\end{problem}

Note that the cost function is the composition of a series of, typically,
nonsmooth mappings.
Such problems have been studied in the operations research literature and
are typically solved by means of cutting plane methods which allow the
solution of problems with short prediction horizons
and linear stage cost functions \cite{Asamov2015,Bruno2016979}.
Here, we employ the method introduced in~\cite{HSBP2017} to decompose
the nested conditional risk mappings and reformulate \Cref{prob:riskAverseMpc:nestedConditionalRiskMappings} as an \ac{miqcqp}.

\subsection{Reformulation as an \acs{miqcqp}}
\label{sec:reformulation}

In this section we employ \eqref{eq:avarEpigraphRelaxation:node:long} to decompose the nested formulation stated above for the case of \ac{avar}.
By doing so, we will cast the overall \ac{mpc} problem as an \ac{miqcqp}.

\begin{theorem}[Problem reformulation]
Define the variables $t^{(i)}$ for all non-leaf nodes $i$
and the variables $\xi^{(i_+)}$ for all nodes $i_+\in\N_{[1,\mu-1]}$.
Define
\[
 \Psi^{(i)} = t^{(i)} {}+{}
	           \textstyle \sum_{i_+\in\child(i)}
	             \tfrac{\pi^{(i_+)}}{\pi^{(i)}}
	                \xi^{(i_+)},
\]
for all non-leaf nodes $i$, which is the cost in the minimisation
in \eqref{eq:avarEpigraphRelaxation:node:long}.
Let the underlying risk measure be the average value-at-risk
at level $\alpha\in[0,1]$.
Then, with the additional decision variables
$\boldsymbol{t} = [t^{(i)}]_{i\in(\N_{[0, \mu-1]}\setminus \nodes(N))}$ and
$\boldsymbol{\xi} = [\xi^{(i)}]_{i\in\N_{[1, \mu-1]}}$,
\cref{prob:riskAverseMpc:nestedConditionalRiskMappings} is equivalent to
the following problem, in the sense that both problems yield equal optimal values.
 \begin{problem}
\label[problem]{prob:riskAverseMpc}
Solve the optimisation problem
\[
	\Minimize_{{\boldsymbol{v}, \boldsymbol{x}, \boldsymbol{q}}, \boldsymbol{t}, \boldsymbol{\xi}}\ \Psi^{(0)}
\]
$\subjectto$
\begin{align*}
	& \phantom{\alpha} \xi^{[i]} \geq 0, \\
	& \alpha \xi^{[i]} \geq Z^{[i]} - t^{(i)}\One_{|\child(i)|}, {}\text{ if } \stage(i) = N-1 \\
	& \alpha \xi^{[i]} \geq Z^{[i]} + \Psi^{[i]} - t^{(i)}\One_{|\child(i)|}, {}\text{ if } \stage(i) < N-1 \\
 	& \text{constraints \eqref{eq:forecastAndScenarioTree:constraints} and given initial conditions } x^{(0)}, \delta_{\mathrm{t}}^{(0_-)} \\
	& \forall i_+\in\N_{[1,\mu-1]} \text{ and }  i = \ancestor(i_+).
\end{align*}
\end{problem}

\begin{proof}
\begin{subequations}\label{eq:phi}
Let us introduce a vector $\Phi \in \R^{\mu - |\nodes(N)|}$, which is defined over the
non-leaf nodes of the scenario tree.
In particular, we associate a value $\Phi^{(i)}\in\R$ to every non-leaf node $i\in \N_{[0, \mu-1]} \setminus \nodes(N)$.
Similar to~\eqref{eq:costVectorAtStageJPlusOne}, we segment $\Phi$ stage-wise into $\Phi_j = [\Phi^{(i)}]_{i\in\nodes(j)}$.
Let us first define $\Phi_{N-1}\in\R^{|\nodes(N-1)|}$ over the set $\nodes(N-1)$ as
\begin{equation}\label{eq:phi:nMinusOne} 
	\Phi_{N-1} = \rho_{N-1}(Z_N).
\end{equation}
$\Phi_{N-1}$ is the conditional value-at-risk at stage $N-1$ and
it corresponds to the innermost term in the nested multistage cost
function \eqref{eq:varRho} in \cref{prob:riskAverseMpc:nestedConditionalRiskMappings}.
Furthermore, let us define $\Phi_j\in\R^{|\nodes(j)|}$ over
the set $\nodes(j)$, $j\in\N_{[0, N-2]}$ as
\begin{equation}\label{eq:phi:nElse} 
	\Phi_{j} = \rho_{j}(Z_{j+1} + \Phi_{j+1}).
\end{equation}
\end{subequations}

The recursive definition \eqref{eq:phi} allows us to express the nested multistage
risk measure \eqref{eq:varRho} as
\begin{align} 
 \label{eq:phi_0_total_cost}
 \varrho_N(Z_1,\ldots, Z_N)
 {}={}&
 \rho_0 (Z_1+ \rho_1(Z_2 + \ldots
 \notag\\
 &\quad \quad + \rho_{N-2}(Z_{N-1}+ \Phi_{N-1})) \ldots )
 \notag \\
\smash{\vdotswithin{ = }} &
 \notag \\
 {}={}&
 \rho_0 (Z_1 + \Phi_1) = \Phi_0.
\end{align}

Note that $\Phi_{N-1}$ in \eqref{eq:phi:nMinusOne} is composed of elements
$\Phi^{(i)}$, $i\in\nodes(N-1)$ and because of \eqref{eq:avarEpigraphRelaxation:node:long},
\begin{equation} 
 \label{eq:Phi_definition_endOfTree}
 \Phi^{(i)}
 {}={}
\rho^{(i)}(Z^{[i]})
  	{}={} \min_{
		\subalign{
		\xi^{[i]}	&\geq{} 0\\
		\alpha \xi^{[i]} &\geq{} Z^{[i]} {}-{} t^{(i)}1_{|\child(i)|} \hspace{-5.5em}
		}}
	    {\Psi^{(i)}},
\end{equation}
where the minimisation is carried out over the decision variables $t^{(i)}$
and $\xi^{[i]} = [\xi^{(i_+)}]_{i_+\in\child(i)}$.

The vector $\Phi_{j}$ in \eqref{eq:phi:nElse} comprises the elements
$\Phi^{(i)}$, $\stage(i) = j$.
For \ac{avar}, as described in \eqref{eq:avarEpigraphRelaxation:node:long} they are
\begin{equation}\label{eq:randomVariablePsiJ} 
	\Phi^{(i)}
	 =
 	\rho^{(i)}(Z^{[i]} + \Phi^{[i]})
	 =
	\min_{
		\subalign{
		\xi^{[i]} &\geq{} 0\\
			\alpha \xi^{[i]} &\geq{} Z^{[i]} {}+{} \Phi^{[i]} {}-{} t^{(i)}\One_{|\child(i)|} \hspace{-7.5em}
		}}
		{\Psi^{(i)}}\hspace{6.5em}
\end{equation}

In light of \eqref{eq:phi_0_total_cost}, the above recursive procedure leads to the formulation of the following optimisation problem.
\begin{problem}\label{prob:riskAverseMpc:Phi}
Solve the optimisation problem
\[
	\Minimize_{{\boldsymbol{v}, \boldsymbol{x}, \boldsymbol{q}}, \boldsymbol{t}, \boldsymbol{\xi}}\
	\Phi^{(0)}
\]
$\subjectto$
\begin{align*}
	& \phantom{\alpha} \xi^{[i]} \geq 0, \\
	& \alpha \xi^{[i]} \geq Z^{[i]} - t^{(i)}\One_{|\child(i)|}, {}\text{ if } \stage(i) = N-1, \\
	& \alpha \xi^{[i]} \geq Z^{[i]} + \Phi^{[i]} - t^{(i)}\One_{|\child(i)|},  {}\text{ if } \stage(i) < N-1, \\
 	& \text{constraints \eqref{eq:forecastAndScenarioTree:constraints} and given initial conditions } x^{(0)}, \delta_{\mathrm{t}}^{(0_-)} \\
	& \forall i_+\in\N_{[1,\mu-1]} \text{ and }  i = \ancestor(i_+).
\end{align*}
\end{problem}

In \Cref{prob:riskAverseMpc:Phi}, we substitute $\Phi^{(0)}$ with $\Psi^{(0)}$ in the cost
function.
Furthermore, because of \Cref{lemma:auxiliary} we can replace the constraint
$\alpha \xi^{[i]} \geq Z^{[i]} + \Phi^{[i]} - t^{(i)}\One_{|\child(i)|}$
by the constraint $\alpha \xi^{[i]} \geq Z^{[i]} + \Psi^{[i]} - t^{(i)}\One_{|\child(i)|}$.
This leads to \cref{prob:riskAverseMpc} completing the proof.
\end{proof}
\end{theorem}

Since the operating cost of \Cref{sec:operatingCosts} involves quadratic functions and binary variables, \Cref{prob:riskAverseMpc} is an \ac{miqcqp}.
As we will demonstrate in \Cref{sec:caseStudy}, this can be solved efficiently by standard software such as CPLEX or Gurobi.

\subsection{Risk-averse MPC}
\label{sec:riskAverseMpcScheme}
The risk-averse optimal control problem is solved
given the current measured state of the system
and a scenario tree that describes the distribution of
future available renewable infeed and demand over $N$ stages.
By solving \Cref{prob:riskAverseMpc} at every time instant $k$ we obtain a set of
control actions $v^{(i)}$ for ${i\in \N_{[0,\mu-1]} \setminus \nodes(N)}$
(see~\Cref{fig:probabilitytreeUnsymmetricVariables}).
Then, the control action associated with the root node of the tree, $v^{(0)}$,
is applied to the system.
This procedure is repeated at every time instant $k\in\N_{0}$ in a receding horizon fashion leading to a risk-averse \acf{mpc} scheme \cite{BM1999}.


\section{Case study}
\label{sec:caseStudy}

In this case study, we aim to demonstrate the properties of the risk-averse \acl{mpc} strategy introduced in \Cref{sec:riskAverseMpcScheme}.
For the simulations, the \ac{mg} in \Cref{fig:examplaryMicrogrid} is used.
It comprises a storage, a conventional and a renewable unit with the parameters from \Cref{tab:caseStudy:ModelParameters} as well as a load.
The units and the load are connected by transmission lines that all have susceptance $b_{ij} = \unit[-20]{pu}$ and conductance $g_{ij} = \unit[2]{pu}$.
Hence, \eqref{eq:nodePowerToLinePower} becomes
\[
	\begin{bmatrix}
		p_{\mathrm{e},1}(k) \\
		p_{\mathrm{e},2}(k) \\
		p_{\mathrm{e},3}(k) \\
		p_{\mathrm{e},4}(k)
	\end{bmatrix}
	=
	\underbrace{%
		\begin{bmatrix}
			1 & 0 & 0 & 0\\
			0 & \nicefrac{-1}{3} & \nicefrac{1}{3} & 0\\
			0 & \nicefrac{2}{3} & \nicefrac{1}{3} & 0\\
			0 & \nicefrac{1}{3} & \nicefrac{2}{3} & 0
		\end{bmatrix}%
	}_{F}
	\begin{bmatrix}
		p_{\mathrm{t}}(k) \\
		p_{\mathrm{s}}(k) \\
		p_{\mathrm{r}}(k) \\
		w_{\mathrm{d}}(k)
	\end{bmatrix}.
\]
Each line can transmit power between $\unit[-1.3]{pu}$ and $\unit[1.3]{pu}$.
Note that all values are given in per-unit ($\unit{pu}$), see, e.g., \cite{KBL1994}.

All simulations were performed in MATLAB~2015a using the closed-loop setup shown in \Cref{fig:systemOverview}.
In what follows, we will discuss the different parts of it, i.e., forecast, scenario reduction, \ac{mpc} and the \ac{mg} plant model.

\begin{figure}
	\centering
	\includegraphics{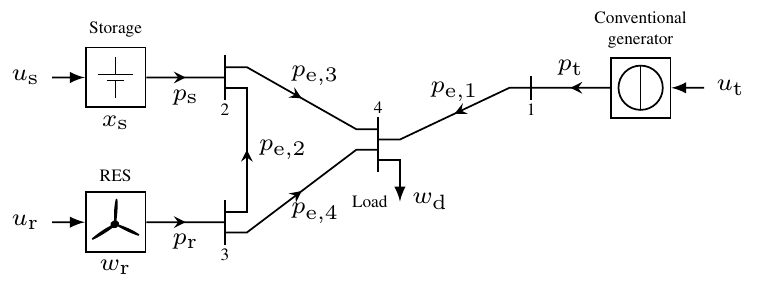}
	\caption{\ac{mg} considered in the simulation case study with storage, renewable and conventional unit as well as load.
	Each unit is connected to a bus, which is connected with other buses by transmission lines.}
	\label{fig:examplaryMicrogrid}
\end{figure}

The wind speed and load forecasts were obtained with the seasonal \ac{arima} models from \cite{HSB+2015}.
These were implemented using MATLAB's Econometrics toolbox.
The available power of the \ac{res}, i.e., of the wind turbine, was calculated from the wind speed forecast using the cubic approximation from~\cite{HSB+2015}.

The scenario tree was generated from a collection of $500$ independent scenarios of load and available renewable power using \cite[Algorithm~5]{OPG+2015}.
The maximum number of children per node was chosen to be $6$ for the root node, $2$ for the nodes of stage $1$ and $1$ for all other stages.
The tree was modified to include \acl{help} scenarios corresponding to the cases of very low available renewable power and very high load and vice versa.
Hence, the scenario tree contains $6 \cdot 2 + 2 = 14$ scenarios.
Because of their very low probability, \acl{help} scenarios have minor influence on the stochastic case at $\alpha = 1$.
However, for smaller values of $\alpha$ that are close to zero, i.e., as we approach the worst-case, their influence increases significantly as will be illustrated later.

\begin{table}
	\centering
	\newcommand{\cIn}[2]{\ensuremath{c_{\mathrm{#1}}^{#2}}}
	\newcommand{\cS}[1]{\cIn{s}{#1}}
	\newcommand{\cT}[1]{\cIn{t}{#1}}
	\newcommand{\cR}[1]{\cIn{r}{#1}}
	\caption{Unit parameters and weights of cost function.}
	\label{tab:caseStudy:ModelParameters}
	\begin{tabular}{clccl}
		\toprule
		Parameter & Value & & Weight & Value\\
		\cmidrule{1-2} \cmidrule{4-5}
		$[{p}_{\mathrm{t}}^{\min}, {p}_{\mathrm{r}}^{\min}, {p}_{\mathrm{s}}^{\min}]$ & $[0.4, 0, -1]\,\unit{pu}$ & & $\cT{}$ & $\unit[0.1178]{}$\\[1pt]
		$[{p}_{\mathrm{t}}^{\max}, {p}_{\mathrm{r}}^{\max}, {p}_{\mathrm{s}}^{\max}]$ & $[1, 2, 1]\,\unit{pu}$ & & $\cT{\prime}$ & $\unit[0.751]{\nicefrac{1}{pu}}$\\[1pt]
		$[{x}^{\min}, {x}^{\max}]$ & $[0, 7]\,\unit{pu\,h}$ & & $\cT{\prime\prime}$ & $\unit[0.0693]{\nicefrac{1}{pu}}$ \\[1pt]
		$[\tilde{x}^{\min}, \tilde{x}^{\max}]$ & $[0.5, 6.5]\,\unit{pu\,h}$ & & $\cR{}$ & $\unit[1]{\nicefrac{1}{pu}}$\\[1pt]
		$x^0$ & $\unit[3]{pu\,h}$ & & $\cS{}$ & $\unit[3 \cdot 10^3]{\nicefrac{1}{pu\,h}}$ \\[1pt]
		$[{K}_\mathrm{t}, {K}_\mathrm{s}]$ & $[1, 1]$ & & $\cT{\mathrm{sw}}$ & $\unit[0.1]{}$ \\
		\bottomrule
	\end{tabular}
\end{table}

For \ac{mpc}, a prediction horizon of $N = 8$, a sampling time of $\samplingTime=\unit[\nicefrac{1}{2}]{h}$ and a discount factor of $\gamma = 0.95$ were chosen.
The different controllers were implemented using YALMIP~R20180612~\cite{Lof2004} and Gurobi~7.5.2 as a numerical solver.
To speed up the computations, the results from the previous iteration were used as initial values to warm-start the optimisation.
Furthermore, the binary switch state of the conventional unit was relaxed for all stages greater than $3$, i.e., $\delta_\mathrm{t}^{(i)}\in [0,1]^T$ for $\stage(i) \geq 4$ in the risk-averse \ac{mpc}.
The simulations in \Cref{sec:nominalSimulations} were performed on a computer with an Intel$^\text{\textregistered}$ Xeon$^\text{\textregistered}$ E5-1620 v2 processor @\unit[3.70]{GHz} with \unit[32]{GB}\,RAM.
Here, the maximum solve time of Gurobi (excluding the time required by YALMIP to parse the problem) was below $\unit[7]{s}$ (see \Cref{tab:CaseStudyComparison}).
Considering a sampling time of $\samplingTime = \unit[\nicefrac{1}{2}]{h}$ (see, e.g., \cite{pbraun2016islanding}), this is adequately~fast.

In the plant model, the efficiency of the storage unit and an \ac{ac} power flow model were considered to obtain more realistic simulation results.
For the storage unit, the model described in \Cref{rem:mgStorageModel} with charging and discharging efficiency ${\eta^{\mathrm{c}} = \eta^{\mathrm{d}} = 0.92}$ and self discharge $x^{\mathrm{sd}} = \unit[2\cdot10^{-3}]{pu\,h}$ was used.
For \ac{ac} power flow, the model from \Cref{rem:mgPowerFlowModel} was used employing the MATLAB function fmincon.
Here, the line parameters posed earlier and voltages $v_i = \unit[1]{pu}$, $i \in \N_{[1,4]}$ were assumed.
Note that this model was only used to simulate the \ac{mg} plant and not in the \ac{mpc} problem formulations.

\begin{subequations}
To compare the outcome of the different simulations we introduce the average economically motivated cost
\begin{equation}\label{eq:averageOperatingCosts}
	\bar{\ell}_{\mathrm{o}} = \textstyle\frac{1}{K}\sum_{k=1}^{K} \ell_{\mathrm{o}}( v_s(k-1), v(k), z(k)).
\end{equation}
over a simulation horizon $K$ with $\ell_{\mathrm{o}}$ from \eqref{eq:operatingCosts:economic}.
Using \eqref{eq:softConstraints:cost}, we also introduce the average cost related to the stored energy,
\begin{equation}\label{eq:averageStorageRelatedCosts}
	\bar{\ell}_{\mathrm{s}} = \textstyle\frac{1}{K}\sum_{k=1}^{K} \ell_{\mathrm{s}}(x(k)).
\end{equation}
\end{subequations}

Having described the simulation setup and the average costs, we can now discuss the simulation results.
Here, we first provide a comparison of different controllers for a nominal simulation run.
Then, we illustrate the properties of the risk-averse \ac{mpc} in more detail in a sensitivity analysis.

\subsection{Nominal simulations}
\label{sec:nominalSimulations}

\begin{figure*}[tb]
	\centering
	\includegraphics{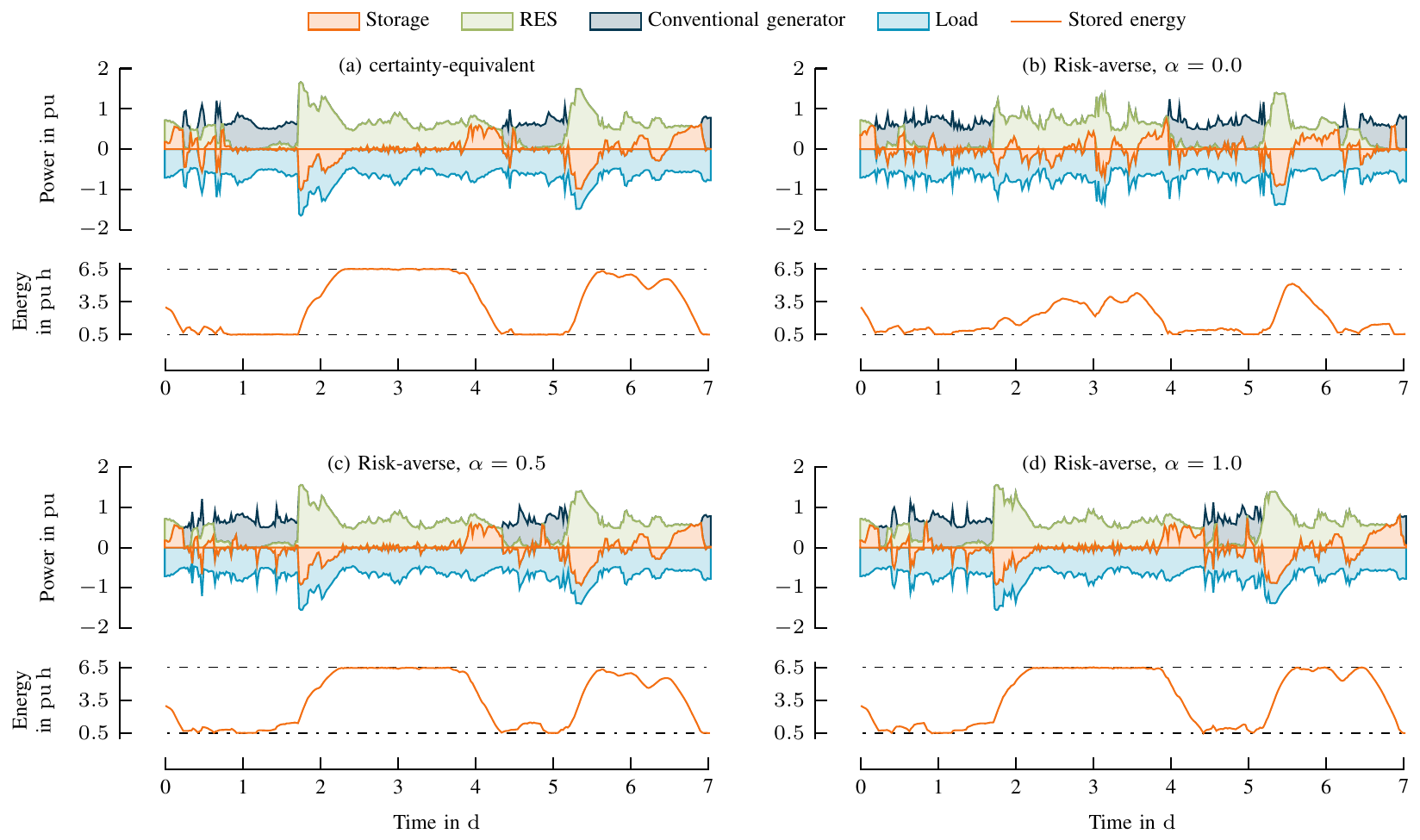}
	\caption{Power and energy of the \ac{mg} using different controllers over one week of simulation with a sampling time of \unit[\nicefrac{1}{2}]{h}.}
	\label{fig:caseStudyResults}
\end{figure*}

\begin{table}[t]
	\renewcommand*{\thefootnote}{\alph{footnote}}
	\centering
	\caption{Running costs, conventional and renewable infeed of closed-loop simulation with simulation horizon $K = 336$.}
	\label{tab:CaseStudyComparison}
	\begin{threeparttable}
	\newcommand{\csvDispColumn}[1]{%
		\csvreader[head to column names, separator=comma, filter equal={\thecsvinputline}{2}]%
			{simulationResults.csv}{}{#1} &
		\csvreader[head to column names, separator=comma, filter equal={\thecsvinputline}{3}]%
			{simulationResults.csv}{}{#1} &
		\csvreader[head to column names, separator=comma, filter equal={\thecsvinputline}{4}]%
			{simulationResults.csv}{}{#1} &
		\csvreader[head to column names, separator=comma, filter equal={\thecsvinputline}{5}]%
			{simulationResults.csv}{}{#1}
		}
	\begin{tabular}{rd{0.0}d{0.0}d{0.0}d{0.0}}
		\toprule
			& \noDecimal{\multirow{3}{1.08cm}{\centering \vspace{-3mm} Certainty-equival.}} & \multicolumn{3}{c}{Risk-averse, $\alpha=$} \\
										\cmidrule{3-5}
			&  & \noDecimal{\parbox{5mm}{\centering $0.0$\tnote{a}}} & \noDecimal{\parbox{5mm}{\centering $0.5$}} & \noDecimal{\parbox{5mm}{\centering $1.0$\tnote{b}}} \\
			\midrule
			Avg. power costs $\bar{\ell}_{\operatorname{o}}$ & \csvDispColumn{\costs} \\
			Avg. energy cost $\bar{\ell}_{\operatorname{s}}$ & \csvDispColumn{\softConstCosts} \\
			Avg. conventional infeed in \unit{pu}& \csvDispColumn{\thermal} \\
			Avg. infeed of RES in \unit{pu} & \csvDispColumn{\renewable} \\
			\midrule
			Constraint violations power & \csvDispColumn{\errPower} \\
			Switching actions & \csvDispColumn{\switchingActions} \\
			\midrule
			Avg. solve time in \unit{s} & \csvDispColumn{\meanSolverTime} \\
			Maximum solve time in \unit{s} & \csvDispColumn{\maxSolverTime} \\
		\bottomrule
	\end{tabular}
	\begin{tablenotes}
		\item[a] Corresponds to worst-case \ac{mpc}.
		\item[b] Corresponds to risk-neutral stochastic (expectation-based) \ac{mpc}.%
	\end{tablenotes}%
	\end{threeparttable}%
\end{table}%

In what follows, the risk-averse \ac{mpc} approach for different values of $\alpha$ and a certainty-equivalent \ac{mpc} approach where the mean value of the forecast is considered as the true value \cite{HNRR2014} are compared in closed-loop simulations.
As shown in \Cref{tab:CaseStudyComparison}, the certainty-equivalent approach leads to power constraint violations.
These are caused by the discrepancy between data-based forecast and actual uncertain input, as well as the mismatch between the model in \ac{mpc} and in the plant simulation and render the approach unsuitable for a safe operation control.
Therefore, it is not discussed further.

For the risk-averse approaches, it can be seen in \Cref{fig:caseStudyResults} that with increasing $\alpha$, energy is stored faster.
Furthermore, the mean infeed from \ac{res} increases and the infeed of the conventional unit decreases (see~\Cref{tab:CaseStudyComparison}).
These two effects contribute significantly to a decrease of the cost $\bar{\ell}_{\mathrm{o}}$.
Thus, with $\alpha = 1$, the cost $\bar{\ell}_{\mathrm{o}}$ is reduced by \unit[14]{\%} compared to $\alpha = 0$.

As the approach becomes more averse to risk with decreasing $\alpha$, the average cost associated with the stored energy, $\bar{\ell}_{\mathrm{s}}$, decreases.
Two effects drive this decrease:
(i)~the handling of \ac{help} scenarios in \ac{mpc} and
(ii)~the mismatch between the \ac{mg} model in the plant simulation and that in \ac{mpc}.
However, simulations where the same model is used for the \ac{mg} plant and \ac{mpc} indicate that effect~(i) plays the dominant role.

\ac{help} scenarios represent extreme combinations of forecast values which are very unlikely to happen.
A misestimation of their probability can have a strong impact on the closed-loop performance.
Risk-neutral stochastic \ac{mpc} ($\alpha=1$) is agnostic to such misestimations and can therefore not act proactively against misestimated \ac{help} events.
In closed-loop simulations, this is reflected by higher energy related costs, $\bar{\ell}_{\mathrm{s}}$, as probability distributions that do not follow the scenario tree are not considered.
With decreasing $\alpha$, ambiguity in the probability distribution is taken more into account leading to \ac{mpc} formulations that account for misestimated probabilities of unlikely events.
This leads to an increasing importance of unlikely scenarios where the energy is outside the interval $[\tilde{x}^{\min}, \tilde{x}^{\max}]$ and
results in control actions where more energy values are inside this interval leading to lower values of $\bar{\ell}_{\operatorname{s}}$.

A comparison of the overall average cost, $\bar{\ell} = \bar{\ell}_{\mathrm{o}} + \bar{\ell}_{\mathrm{s}}$, shows that the lowest value $\bar{\ell} = 2.99$ is achieved for $\alpha = 0.5$.
This is about \unit[6]{\%} lower than the overall cost for $\alpha = 0$, $\bar{\ell} = 3.2$, and about \unit[26]{\%} lower than the overall cost for $\alpha = 1$, $\bar{\ell} = 4.09$.

The nominal simulations show that the risk-averse \ac{mpc} together with an appropriate construction of the scenario tree leads to suitable operation strategies for islanded \acp{mg}.
In practice, this translates into a desirable trade-off between performance and robustness to uncertain probability distributions.

\subsection{Sensitivity analysis}

To illustrate the performance of the risk-averse approach in presence of inaccurate forecasts, a sensitivity analysis was carried out.
In the analysis, $1000$ closed-loop simulations were performed for the first three days, i.e., $K = 144$ simulation steps, of the scenario shown in \Cref{fig:caseStudyResults}.
To demonstrate the robustness of the risk-averse approach to uncertainty in the probability distribution, we added noise to the uncertain input of the \ac{mg} plant model.
This way, occasional \ac{help} events are artificially added to illustrate the positive effects of the risk-averse \ac{mpc} approach.
In the following analysis, we consider two different probability distributions of the additional noise:
(i)~constant offset in the mean value, and
(ii)~occasional offset that randomly occurs in \unit[10]{\%} of the simulation steps.

\subsubsection{Disturbance with constant offset}
\label{sec:disturbanceWithConstantOffset}

Here, a Gaussian noise term with nonzero mean and standard deviation equal to that of the \ac{arima} forecast training residuals is added to the wind speed and load time series.
In particular, Gaussian noise with mean $\unit[0.048]{pu}$ and standard deviation $\unit[0.032]{pu}$ was added to the load and Gaussian noise with mean $\unitfrac[-0.795]{m}{s}$ and standard deviation $\unitfrac[0.53]{m}{s}$ was added to the wind speed before the available wind power, $w_{\operatorname{r}}$, was obtained.
The different scenarios of wind and load are shown in \Cref{fig:windAndLoadMeanOffset}.

\begin{figure}[tb]
	\centering
	\includegraphics{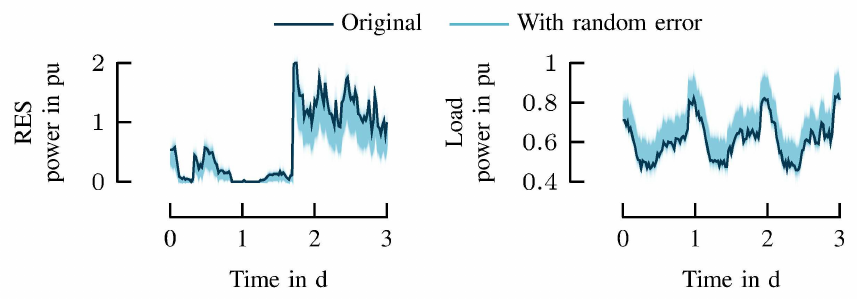}
	\caption{1000 scenarios of wind and load data used in sensitivity analysis with constant offset of 1.5 times the standard deviation.}
	\label{fig:windAndLoadMeanOffset}
\end{figure}

\begin{figure}[tb]
	\centering
	\includegraphics{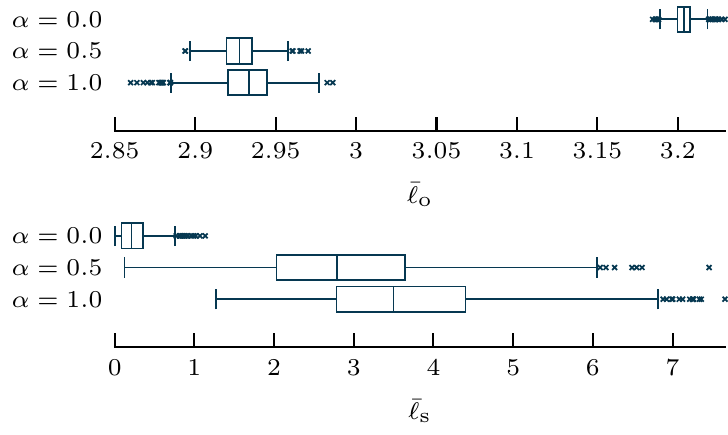}
	\caption{Results of sensitivity analysis with systematic error, i.e., constant offset in the mean value of the additional disturbance.}
	\label{fig:boxplotObjectiveConstantErrors}
\end{figure}

The resulting distribution of $\bar{\ell}_{\mathrm{o}}$ is illustrated in \Cref{fig:boxplotObjectiveConstantErrors}.
It can be observed that the mean of the economically motivated costs first decrease from $\alpha = 1$ to $\alpha = 0.5$ by \unit[0.14]{\%}.
Furthermore, the standard deviation of $\bar{\ell}_{\mathrm{o}}$ significantly decreases from $0.019$ for $\alpha = 1$ to $0.012$ for $\alpha = 0.5$.
Then, for $\alpha = 0$ the mean of $\bar{\ell}_{\mathrm{o}}$ increases by \unit[9.3]{\%} due to the increased conservativeness of the robust \ac{mpc} approach.
This shows that choosing $\alpha < 1$ can protect the system from \ac{help} events, i.e., prevent extreme costs, and even reduce closed-loop costs.
This is also reflected in the lower standard deviation of cost for smaller values of $\alpha$.
Furthermore, the energy related costs $\bar{\ell}_{\mathrm{s}}$ decrease for smaller values of $\alpha$.
Thus, by choosing $\alpha$ appropriately, the costs and the conservativeness of the control scheme can be tuned.
This adds an important degree of freedom to the traditional design procedures of worst-case ($\alpha = 0$) and stochastic ($\alpha = 1$) approaches.

\subsubsection{Disturbance with occasional extreme events}

Here, the mean value of the additional noise was chosen different from zero for only $\unit[10]{\%}$ of the data points to model occasional extreme events.
For the other data points, the mean of the additional noise was set to zero.
The random nonzero offset in \unit[10]{\%} of the cases was $\unit[0.096]{pu}$ for load and $\unitfrac[1.589]{m}{s}$ for wind speed.
A standard deviation of $\unit[0.032]{pu}$ for load and $\unitfrac[0.53]{m}{s}$ for wind speed was considered for all cases based on the training residuals of the \ac{arima} forecast models.
This led to the scenarios shown in \Cref{fig:windAndLoadOccExtremeEvents}.

\begin{figure}[t]
	\centering
	\includegraphics{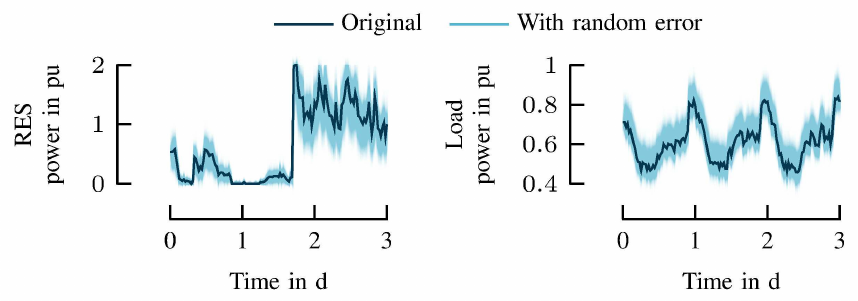}
	\caption{1000 scenarios of wind and load data used in sensitivity analysis with occasional extreme events in \unit[10]{\%} of the cases.}
	\label{fig:windAndLoadOccExtremeEvents}
\end{figure}

\begin{figure}[t]
	\centering
	\includegraphics{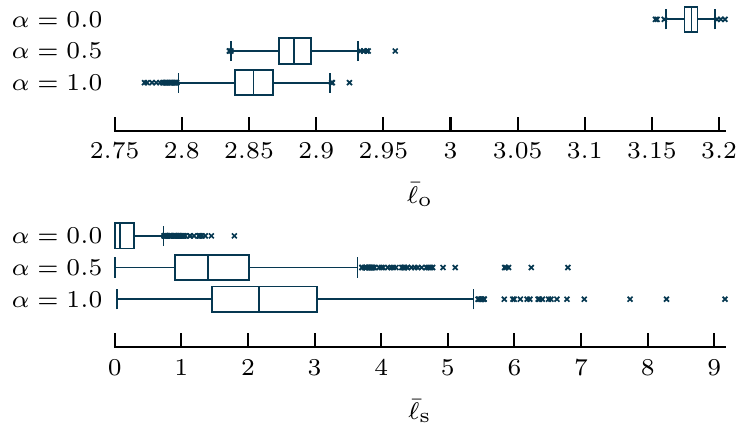}
	\caption{Results of sensitivity analysis with occasional extreme events in \unit[10]{\%} of the cases.}
	\label{fig:boxplotObjectiveExtremeErrors}
\end{figure}

As shown in \Cref{fig:boxplotObjectiveExtremeErrors}, the energy related costs $\bar{\ell}_{\mathrm{s}}$ decreases for lower values of $\alpha$, yet, lower economically motivated costs $\bar{\ell}_{\mathrm{o}}$, are not observed as~$\alpha$ decreases.
However, the standard deviation of $\bar{\ell}_{\mathrm{o}}$ decreases from $0.024$ for $\alpha = 1$ to $0.007$ for $\alpha = 0$, indicating that the average cost for operating the grid becomes less sensitive to inaccurate probability distributions with decreasing $\alpha$.
This shows that the importance of \ac{help} events can be explicitly parametrised by adapting $\alpha$.
Additionally, for a given \ac{mg} setup, the designer of the \ac{mpc} has a tuning knob to strike a suitable trade-off between the economically motivated $\bar{\ell}_{\mathrm{o}}$ and the state related $\bar{\ell}_{\mathrm{s}}$.


\section{Conclusions}
\label{sec:Conclusions}

In this work we presented a risk-averse \ac{mpc} strategy for islanded \acp{mg} with very high share of \ac{res} which enables us to trade economic performance for safety by interpolating between worst-case and risk-neutral stochastic formulations.
The approach provides resilience with respect to misestimations of the underlying probability distributions of demand and available renewable infeed.
Therefore, it is suitable for practical implementations, where these distributions are not known exactly or change over time.
It also allows for the use of simple and therefore computationally less expensive scenario trees at the expense of operating with a slightly more conservative regime.
Furthermore, the presented \ac{mpc} scheme is able to protect \acp{mg} against \acl{help} events such as sudden drops of available renewable power or unexpected increase of demand.
Finally, the proposed risk-averse \ac{mpc} formulation can be cast as an \ac{miqcqp} which can be solved by commercial solvers as indicated in \Cref{sec:riskAverseMpc}.

For a large number of integer variables the problem complexity can potentially become prohibitive.
In the future we would like to look into this topic in more detail by considering \acp{mg} with more conventional and renewable units as well as scenario trees with a higher number of nodes.
Furthermore, as the operation regime is significantly influenced by the state of charge, we plan to consider more complex storage dynamics.
Future work will also address chance constraints and decreasing the solver time by devising paralleliseable optimisation algorithms (see, e.g, \cite{SSBP2017}) that can run on graphics processing units.
Also, we plan to study statistically meaningful ways to choose the level of risk aversion (see, e.g.,~\cite{SSP2019}).

\appendices

\section{Auxiliary Results}

\begin{lemma}\label[lemma]{lemma:auxiliary}
 Let $\emptyset{}\neq{}\mathbb{X}\subseteq\R^n$, $\emptyset{}\neq{}\mathbb{Y}{}\subseteq\R^m$ and for every $x\in{}\mathbb{X}$,
 $f(x,y)$ attains a minimum over $\mathbb{Y}$, i.e., $\min_{y\in\mathbb{Y}}f(x,y)$ exists.
 Then, the optimisation problem
 \begin{align*}
	\Minimize_{x\in \mathbb{X}, v} F(x,v) \quad & \subjectto \quad \min_{y\in \mathbb{Y}} f(x,y) \leq \beta, \\
	\intertext{with cost function $F:\R^n\times\R^m \rightarrow \R$ is equivalent to}
  	\Minimize_{x \in \mathbb{X},v} F(x,v) \quad & \subjectto \quad y\in \mathbb{Y},~f(x,y) \leq \beta.
 \end{align*}
\end{lemma}
\begin{proof}
 As the two problems have the same cost function, it suffices to show that
 they have the same constraint sets.
 Therefore, we define
 \(
      \mathbb{S}
  {}={}
      \big\{x\in\R^n {}\mid{} \min_{y\in \mathbb{Y}}f(x,y)\leq \beta\big\}
 \)
 and
 \(
      \mathbb{S}'
  {}={}
      \{x \in\mathbb{R}^n {}\mid{} \exists y \in \mathbb{Y} \text{ such that } f(x,y)\leq \beta\}
 \).

	Take $x\in \mathbb{S}$, i.e., $\min_{y\in \mathbb{Y}}f(x,y)\leq \beta$.
 	Since the minimum exists, there is a $y^\star\in \mathbb{Y}$ such that $f(x,y^\star)\leq \beta$.
 	Hence, $x\in \mathbb{S}'$ and consequently $\mathbb{S}\subseteq \mathbb{S}'$.

	Take $x\in \mathbb{S}'$, i.e., there is a $y_0\in \mathbb{Y}$ such that $f(x,y_0)\leq \beta$.
	Then, $x\in \mathbb{S}$ because $\min_{y\in \mathbb{\mathbb{Y}}}f(x, y) \leq f(x,y_0) \leq \beta$,
	and consequently $\mathbb{S}'\subseteq \mathbb{S}$.
	This proves that $\mathbb{S}'=\mathbb{S}$.
\end{proof}%

\bibliographystyle{IEEEtran}
\bibliography{databasePaper}

\begin{IEEEbiography}[{\includegraphics[width=1in,height=1.25in,clip,keepaspectratio]{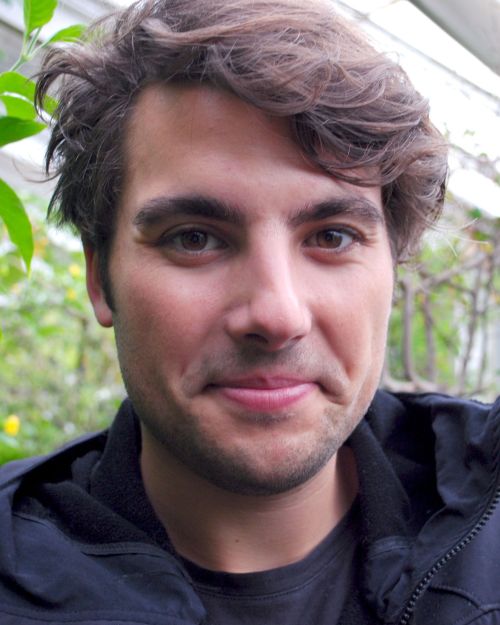}}]{Christian A. Hans} (S’13) is research assistant at the Control Systems Group in the EECS Department at TU Berlin.
Prior to that, he engaged in control issues in low inertia microgrids as development engineer at Younicos AG, Germany.
His research focuses on the application of control methods to power systems with very high share of renewable energy sources.
Here, he is especially interested in
model predictive operation control and time-series forecasting
as well as
distributed and decentralized control.
\end{IEEEbiography}

\begin{IEEEbiography}[{\includegraphics[width=1in,height=1.25in,clip,keepaspectratio]{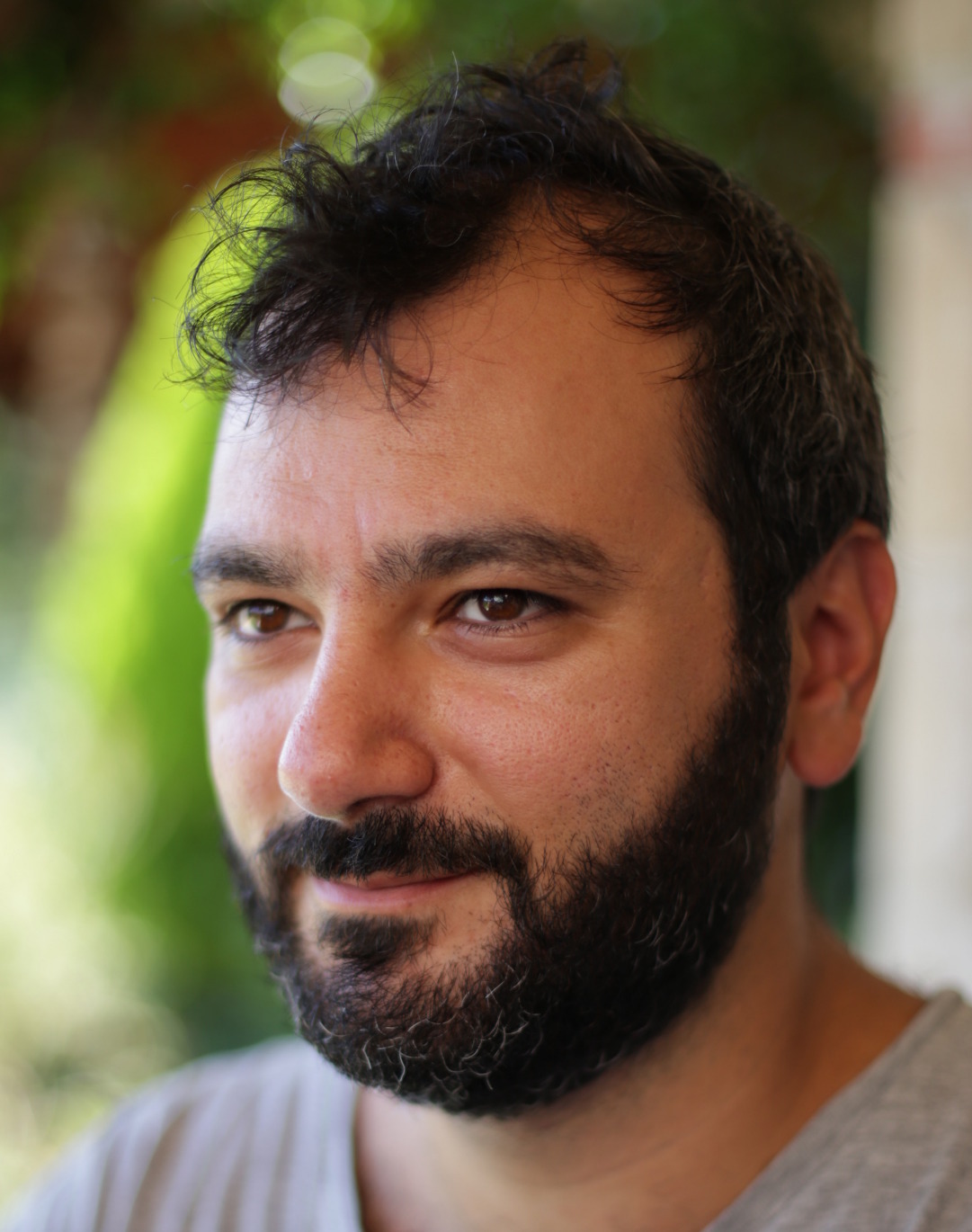}}]{Pantelis Sopasakis} is a lecturer at School of Electronics, Electrical Engineering and Computer Science (EEECS) at Queen's University Belfast and a member of Centre for Intelligent Autonomous Manufacturing Systems (i-AMS).
He received a diploma (MEng) in Chemical Engineering in 2007 and an MSc with honours in Applied Mathematics in 2009 from National Technical University of Athens.
In December 2012, he defended his PhD thesis at School of Chemical Engineering, NTU Athens.
He has held postdoc positions at IMT Lucca (Italy, 2013-16), KU Leuven, ESAT (Belgium, 2016-18) and University of Cyprus, KIOS (Cyprus, 2018).
His research interests focus on the development of model predictive control methodologies, numerical algorithms for large-scale stochastic systems and massive parallelisation on general-purpose graphics processing units (GP-GPUs).
The outcomes of his research have several applications in automotive and aerospace, microgrids, water distribution networks and optimal, safe drug administration.
\end{IEEEbiography}

\begin{IEEEbiography}[{\includegraphics[width=1in,height=1.25in,clip,keepaspectratio]{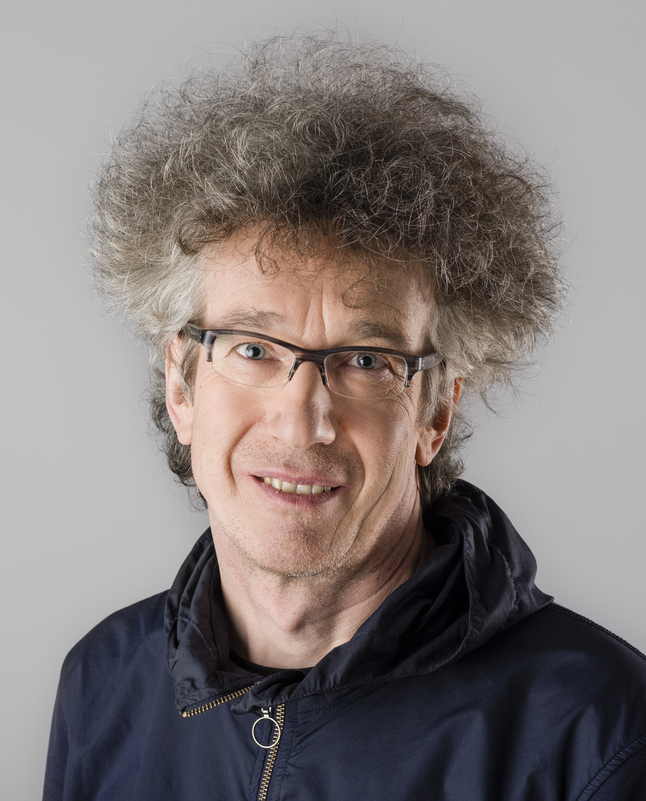}}]{J\"org Raisch} studied Engineering Cybernetics at Stuttgart University and Control Systems at UMIST, Manchester.
He received a PhD and a Habilitation degree, both from Stuttgart University.
He holds the chair for Control Systems in the EECS Department at TU Berlin, and he is also an external scientific member of the Max Planck Institute for Dynamics of Complex Technical Systems.
His main research interests are hybrid and hierarchical control, distributed cooperative control, and control of timed discrete-event systems in tropical algebras, with applications in chemical, medical, and power systems engineering.
\end{IEEEbiography}

\begin{IEEEbiography}[{\includegraphics[width=1in,height=1.25in,clip,keepaspectratio]{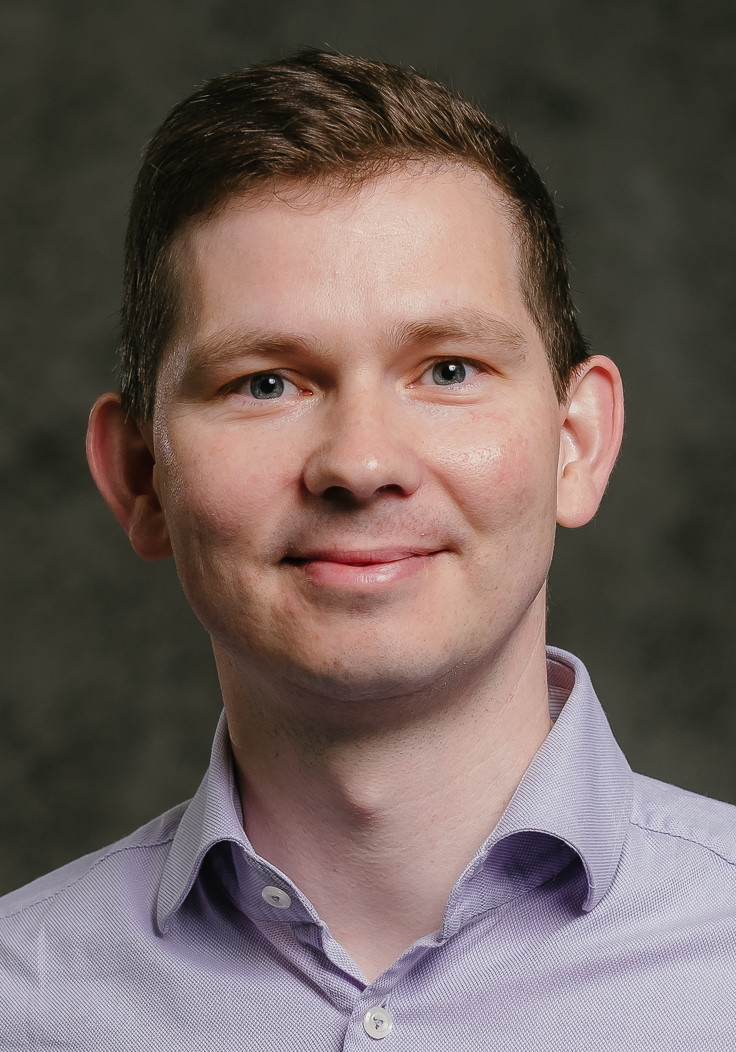}}]{Carsten Reincke-Collon} is Director of Future Technology within the Global Products and Technology unit of Aggreko plc, a leading specialist provider of power, temperature control and energy services.
He studied Electrical Engineering at TU Dresden, Germany, and received his Engineering Diploma degree in 2005.
As external doctoral candidate at the Laboratory of Control Theory of TU Dresden his research focused on invariant feedback design for nonlinear (finite-dimensional) control systems possessing Lie symmetries employing a differential-geometric approach to symmetries of nonlinear systems.
He received his doctorate degree (summa cum laude) in 2011 and joined Younicos AG – a Berlin-based pioneer in battery storage systems and hybrids – serving in leading engineering roles and as member of the senior management team before the acquisition by Aggreko in 2017.
Among his current research interests are solutions for optimal control of high-penetration microgrids.
\end{IEEEbiography}

\begin{IEEEbiography}[{\includegraphics[width=1in,height=1.25in,clip,keepaspectratio]{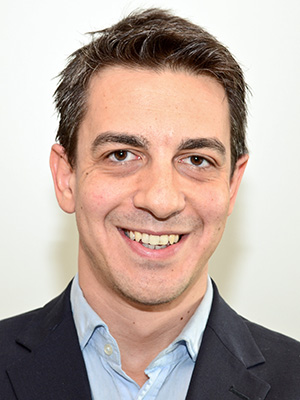}}]{Panagiotis (Panos) Patrinos} (M’13) is assistant professor at the Department of Electrical Engineering (ESAT) of KU Leuven, Belgium. During 2014 he held a visiting assistant professor position in the department of Electrical Engineering at Stanford University. He received his PhD in Control and Optimisation, M.S. in Applied Mathematics and M.Eng. from the National Technical University of Athens in 2010, 2005 and 2003, respectively. He has held postdoc positions at the University of Trento and IMT Lucca, Italy, where he became an assistant professor in 2012.

His research interests lie in the interface of optimisation, learning and control with a broad range of applications including autonomous vehicles, smart grids, water networks, aerospace, multi-agent systems, signal processing and machine learning.

He was the chair of the 38th Benelux meeting on Systems and Control (2019) and co-chair of the 4th European Conference on Computational Optimization (EUCCO 2016).
\end{IEEEbiography}

\end{document}